\documentclass[11pt, leqno, twoside]{article}

\usepackage{amssymb}
\usepackage{amsmath}
\usepackage{amsthm}
\usepackage{amsfonts}
\usepackage{mathrsfs}
\usepackage{indentfirst}
\usepackage{graphicx}
\usepackage{txfonts}

\usepackage{anysize}

\usepackage{fancyhdr}

\usepackage{color}
\usepackage{setspace}

\usepackage{hyperref}
\hypersetup{colorlinks=true, linkcolor=blue, citecolor=blue}

\allowdisplaybreaks

\usepackage{txfonts}

\def\rr{{\mathbb R}}
\def\rn{{{\rr}^n}}

\def\zz{{\mathbb Z}}

\def\nn{{\mathbb N}}

\def\cm{{\mathcal M}}

\def\cd{{\mathcal D}}

\def\ch{{\mathcal H}}

\def\fz{\infty}

\def\dz{\delta}

\def\lz{\lambda}

\def\lf{\left}
\def\r{\right}

\def\ls{\lesssim}
\def\gs{\gtrsim}

\def\noz{\nonumber}
\def\wz{\widetilde}

\def\cs{{\mathcal S}}

\def\dlimsup{\displaystyle\limsup}

\def\loc{{\mathop\mathrm{\,loc\,}}}

\def\q1{\wz q}
\def\Q1{q_1}

\def\loc{{\mathop\mathrm{loc\,}}}

\newtheorem{thm}{Theorem}[section]
\newtheorem{prop}[thm]{Proposition}
\newtheorem{lem}[thm]{Lemma}
\newtheorem{cor}[thm]{Corollary}

\theoremstyle{definition}
\newtheorem{defn}[thm]{Definition}
\newtheorem{rem}[thm]{Remark}

\numberwithin{equation}{section}

\begin{document}
\title{\bf\Large Weak-$L^\infty$ and BMO with Hausdorff Content
\footnotetext{\hspace{-0.35cm}
2020 {\it Mathematics Subject Classification}.
{Primary 42B25.} \endgraf
{\it Key words and phrases.} Hausdorff content, Marcinkiewicz interpolation, Decreasing rearrangement,
Weak-$L^\infty$ space, BMO space. \endgraf}}
\author{Hao Guo, Long Huang, Yangzhi Zhang\footnote{Corresponding author,\
E-mail: yzzhang@hunnu.edu.cn/{\color{red}{October 23, 2026}}}\ \
and Ciqiang Zhuo
}
\date{ }
\maketitle
	
\vspace{-0.6cm}

\begin{center}
\begin{minipage}{13cm}
{\small {\bf Abstract}\quad
The space of weak-$L^{\infty}$ on measure space was introduced by Bennett et al. (Ann. of Math., 113(1981), 601-611)
and proved to contain the space of $L^\infty$ and possess the appropriate interpolation property.
In this paper, we are aimed to continue this line of research on the Euclidean space $\mathbb R^n$ endowed with
Hausdorff content, which is a low-dimensional outer measure on $\mathbb{R}^n$, and fails to be additive or absolutely continuous with respect to Lebesgue measures.
But low-dimensional subsets, such as fractals, of $\mathbb{R}^n$ are Hausdorff capacitable.
Precisely, in the setting of Hausdorff content, an analogue of weak-$L^\infty$ is defined and proved to
possess an improved Marcinkiewicz interpolation theorem.
Moreover, we show that the space of functions with bounded mean oscillation associated to the Hausdorff content
is embedded into this new weak-$L^\infty$ space by establishing a key estimate for the distribution via a packing condition.
  As a by-product,
 the boundedness of the Hardy-Littlewood operator on the space of weak-$L^{\infty}$ is established.}
\end{minipage}
\end{center}

\vspace{0.2cm}

\section{Introduction}
C. Bennett, R.A. Devore and R. Sharpley in \cite{BDS-1} introduced the weak-$L^\fz$ space and studied its relation to the space of functions of
bounded mean oscillation. The motivation for introducing such weak-$L^\fz$ space comes from
the probable limitation of the classical Marcinkiewicz interpolation theorem
associated to an operator $T$, which is required to be bounded from $L^\fz$ to $L^\fz$ and, consequently, fails
to apply directly on operators that are unbounded on $L^\fz$ (e.g. Hilbert transform).

Following \cite{BDS-1}, the weak-$L^\fz$ space $wL^\fz(X)$ with a give measure space $(X,\mu)$ is
defined to be the collection of all measurable
functions $f$ on $X$ whose decreasing rearrangement $f^\ast$ is finite on $(0,\fz)$ and satisfies
$$\|f\|_{wL^\fz(X)}:=\sup_{t\in(0,\fz)}\lf[\frac1t \int_0^t f^\ast(s)\,ds-f^\ast(t)\r]<\fz,$$
where $f^\ast(t):= \inf\{\lz\in(0,\fz):\ \mu(\{x\in X:\ |f(x)|>\lz\})\le t\}$.
This space properly contains $L^\fz$, especially, when $X=\rn$ and $\mu$ is the Lebesgue measure.
Although $wL^\fz(X)$ is not even a linear space, it possesses a improved Marcinkiewicz interpolation property
(see \cite[Theorem 2.2]{BDS-1}).
Precisely, let $T$ be a linear operator $T$ of restricted weak type $(1,1)$ and of restricted weak type $(\fz,\fz)$,
namely, $T$ maps all characteristic functions $f=\textbf{1}_E$ of sets $E\subset X$
with $\mu(E)<\fz$ from $L^1(X)$ and $L^\fz(X)$, respectively, into
weak-$L^1$ space $wL^1(X)$ and weak-$L^\fz$ space $wL^\fz(X)$.
Then $T$ has a unique extension to a bounded operator on $L^p(X)$ with $p\in(1,\fz)$.
This interpolation theorem applies well to operators which are not bounded from $L^\fz(X)$ into $L^\fz(X)$, such as the Hilbert transform as well as the maximal Hilbert transform and,
more generally to the maximal operators associated with arbitrary Calder\'on-Zygmund singular integrals (see \cite{BDS-1}).

Moreover, let $X$ be specified to a cube $Q\subset \rn$ (with sides parallelling to the coordinate axes)
endowed with the Lebesgue measure. Then the weak-$L^\fz$ space $wL^\fz(Q)$ provides a characterization of the class BMO($Q$) of functions of bounded mean oscillation on the cube $Q$ (see \cite[Theorem 3.1]{BDS-1}).
Particularly, the following conclusion holds true: for any given cube $Q$ of $\rn$,
\begin{equation}\label{estimate-1}
\|f\|_{wL^\fz(Q)}\le C_{(n)} \lf\{\|f\|_{{\rm BMO}(Q)}+\frac1{|Q|}\int_Q|f(x)|\,dx\r\},
\end{equation}
where $C_{(n)}$ is a positive constant depending only on the dimension $n$,
and if $\frac1{|Q|}\int_Q|f(x)|\,dx \to 0$
as $|Q|\to \fz$, then
\begin{equation}\label{estimate-2}
\|f\|_{wL^\fz(\rn)}\le c_{(n)}\|f\|_{{\rm BMO}(\rn)}.
\end{equation}
As shown in \cite{Ko07}, the above work also implies that the decreasing rearrangement $f^*$ of
a function $f$ in ${\rm BMO}(\rn)$ belongs to ${\rm BMO}(\rr_+)$ with \(\mathbb{R}_+:=(0,\infty)\) and
\begin{equation}\label{estimate-3}
\|f^\ast\|_{{\rm BMO}(\rr_+)}\le C_{(n)}\|f\|_{{\rm BMO}(\rn)}.
\end{equation}
The explicit bound on the constant arising from \cite{BDS-1} is $C_{(n)}\le 2^{n+5}$, which was further improved to be
$C_{(n)}\le 2(1+2\sqrt{n-1})$ by Burchard-Dafni-Gibara \cite{BDG22} .
The inequalities \eqref{estimate-1}, \eqref{estimate-2} and \eqref{estimate-3} allow us to reduce some extremal problems involving BMO functions on high-dimensional spaces to those for functions of a single variable. This in turn leads to important applications
in rearrangement-invariant spaces, which are significant from the viewpoint of harmonic analysis and mathematical physics.
Recently, Jiao et al. \cite{JOWZ-1} have successfully extended the work of Bennett-Devore-Sharply in \cite{BDS-1} into the
framework of noncommutative analysis.

Motivated by the increasing studying of harmonic analysis on fractals of Euclidean spaces, the Hausdorff content, low-dimensional outer measure on $\rn$, and the associated Choquet integral have received much attention. Some central topics, such as the boundedness of Hardy-Littlewood maximal operator, Muckenhoupt weights and weighted inequality associated the maximal operator, and BMO space, have been extensively studied in the setting of Hausdorff content and the Choquet integral; see, for example, \cite{0702-4,HZZ-2,0923-7,t11}. Particularly, for a given cube $Q$ in $\rn$ and $\delta\in (0,n]$, the \emph{space of functions with bounded $\delta$-dimensional mean oscillation} on $Q$ is defined by setting
\[{\rm {BMO}}^\delta(Q):=\left\{f\in L^1(Q, \ch_{\infty}^{\delta}): \|f\|_{{\rm {BMO}}^\delta(Q)}< \infty\right\},\]
where
\[\|f\|_{{\rm {BMO}}^\delta(Q)}:=\sup_{P\subset Q}\inf_{c\in \mathbb R}\frac{1}{\ch_{\infty}^{\delta}(P)}\int_{P}|f(x)-c|\,d\ch^{\delta}_{\infty}\]
with the supremum being taken over all cubes $P\subset Q$. Here $\ch^{\delta}_{\infty}$ denotes the Hausdorff content (see Section \ref{s2} for the definition). This space was characterized in \cite{HZZ-2} in terms of integral averages and Hausdorff content \(A_p\) weights, which coincides with the classical results at $\delta=n$.
Then the space \(\mathrm{BMO}^\delta(\mathbb{R}_+)\) is defined as a similar way by restricting all cubes to lie in \(\mathbb{R}_+\).

Based on \cite{BDS-1,JOWZ-1}, it is natural to ask a question that, in the setting of Hausdorff content, whether there is a meaningful
counterpart of the weak-$L^\fz$ space,
which in turn allows us not only to improve the interpolation theorem
to make it applicable to more operators,
but also to establish the analogues of the estimates \eqref{estimate-1}, \eqref{estimate-2} and \eqref{estimate-3}
for functions in ${\rm {BMO}}^\delta(Q)$.

The purposes of this paper is aimed to answer this question in the affirmative.
To this end, we begin with some notation and notions.
For a cube $Q$ in $\rn$, we use $wL^{\fz}(Q,\ch_\fz^\delta)$ to denote
the \emph{weak-$L^\fz$ space} corresponding to the Hausdorff content $\ch_\fz^\delta$
(For the specific definition, see Definition \ref{def-w}).
Following \cite{BDS-1}, under the setting of Hausdorff content, a sublinear operator $T$ is said to be of \emph{restricted weak type} $(1,1)$
if there exists a positive constant $C$ such that, for all characteristic functions $f=\textbf{1}_E$ of sets $E\subset\rn$
with $\ch_\fz^\delta(E)<\fz$,
\begin{align*}
	\sup_{t\in(0,\fz)}t\,\ch_\fz^\delta\lf(\lf\{x\in \rn:\ |Tf(x)|>t\r\}\r)\le C\|f\|_{L^1(\rn,\ch_\fz^\delta)};
\end{align*}
and $T$ is said to be of \emph{restricted weak type} $(\fz,\fz)$ if there exists
a positive constant $C$ such that, for all characteristic functions $f$,
$$\|Tf\|_{wL^\fz(\mathbb{R}^n, \mathcal{H}^{\delta}_{\infty})}
\le C\|f\|_{L^\fz(\mathbb{R}^n, \mathcal{H}^{\delta}_{\infty})}.$$

Then we state the following interpolation theorem.

\begin{thm}\label{tm0209-1}
Let $T$ be a sublinear operator of restricted weak type $(1, 1)$ and restricted weak type $(\infty,\infty)$.
Then it holds true that, for all simple functions $f$,
\begin{align}\label{eq0209-4}
		\|Tf\|_{L^p(\rn,\ch^\delta_\infty)}\leq C\|f\|_{L^p(\rn,\ch^\delta_\infty)},\qquad \forall\,p\in(1,\fz),
\end{align}
where $C$ is a positive constant independent of $f$.
In particular, if $T$ is linear, then $T$ has a unique extension to a bounded linear operator on
$L^p(\rn,\ch^\delta_\infty)$ for any $p\in(1,\fz)$.
\end{thm}

The main difficulty for the proof of Theorem \ref{tm0209-1} is to prove the subadditivity of the second-order decreasing rearrangement $f_{\ch_\fz^\delta}^{\ast\ast,E}$ (see \eqref{eq0902} for the definition).
The original argument for the case of Lebesgue measure
strongly relies on the additivity of the Lebesgue measure, which is unavailable for the setting of the Hausdorff content.
Our method is to reduce the subadditivity of the second-order decreasing rearrangement to
a characterization of the decreasing rearrangement associated with the dyadic Hausdorff content (see Proposition \ref{prop59-1}).
Moreover, we also prove the density of the collection of simple functions in $L^p(\rn,\ch_\fz^\delta)$, which has
interest of its own.
This theorem is expected to apply to the Hilbert transform and, more generally, the Calde\'on-Zygmund singular integrals.
However, this remains challenging at present and we maight consideration in future work.

The second main result of this paper is the following theorem,
which extends the estimates \eqref{estimate-1}, \eqref{estimate-2} and \eqref{estimate-3} of \cite{BDS-1} to the setting of Hausdorff content.

\begin{thm}\label{tm1127-1}
Let $\dz\in (0,n]$. Then there exists a positive constant $C$ such that,
for any cube $Q\subset \rn$ and any $f\in {\rm {BMO}}^\delta(Q)$,
\begin{align*}
\|f\|_{wL^{\fz}(Q,\ch_\fz^\delta)}\leq C_{(\delta,n)}\lf(\|f\|_{{\rm {BMO}}^\delta(Q)}+\frac{1}{\ch_{\fz}^{\dz}(Q)}\int_{Q}|f(x)|\,d\ch_{\fz}^{\dz}\r),
\end{align*}
where $C_{(\delta,n)}$ is a positive constant only depending on $\delta$ and $n$.
Furthermore, if $f\in {\rm {BMO}}^\delta(\rn)$ and $\frac{1}{\ch_{\fz}^{\dz}(Q)}\int_{Q}|f(x)|\,d\ch_{\fz}^{\dz}\to 0$
as $\ch_\fz^\delta(Q)\to \fz$, then
$$\|f\|_{wL^{\fz}(\rn,\ch_\fz^\delta)}\le C_{(\delta,n)} \|f\|_{{\rm BMO}^\delta(\rn)}$$
and hence
$$\|f^\ast_{\ch_{\fz}^{\dz}}\|_{{\rm BMO}^\delta(\rr_+)}\le C_{(\delta,n)} \|f\|_{{\rm BMO}^\delta(\rn)}$$
\end{thm}

Theorem \ref{tm1127-1} infers that, within the framework of Hausdorff content, the embedding relationship of  the space
${\rm{BMO}}^{\dz}(Q)$ into the weak-$L^\fz$ space $L^\infty(Q,\ch_{\fz}^{\dz})$ remains valid. This is a framework beyond measure theory.
The proof of Theorem \ref{tm1127-1} is different from that in the classical case,
even thought similar to \cite{BDS-1} we also need to prove a critical estimate \eqref{eq1123-1} below.
In our proof, we need to overcome the difficulty arising from the absence of linearity on Choquet integral and the additivity of Hausdorff content.
To this end, we split the integral to be estimated in \eqref{eq1126-10} into four parts based on
a covering lemma (see Lemma \ref{lm1121-1} below) and a packing condition (see Lemma \ref{lm0903} below).

The organization of this paper is as follows. In Section \ref{s2}, we provide some basic properties of the Hausdorff content
and decreasing rearrangement. In Section \ref{Sec3}, we prove Theorems \ref{tm0209-1}. In Section \ref{Sec4},
we give the proof of Theorem \ref{tm1127-1}.
As an application, in Section \ref{Sec5}, the boundedness of Hardy-Littlewood maximal on the weak-$L^\fz$ space is obtained.

Throughout the paper, the notation $f\ls g$ (resp. $f\gs g$) means $f\le Cg$ (resp. $f\ge Cg$)
for a positive constant $C$ independent of the main parameters, and
$f\sim g$ amounts to $f\ls g\ls f$. We also use the symbol $C_{(\alpha,\,\beta,\dots)}$ to denote a positive constant
which depends on the parameters $\alpha,\,\beta,\dots$, but may vary line to line.
The symbol ${\rm sgn} f$ represents the \emph{sign function} of $f$.
For any subset $E\subset\mathbb R^n$, denote by $\mathbf{1}_E$ the \emph{characteristic function} of set $E$ and,
for any $\delta\in(0,n]$, a function $f$ on $E$ is called a \emph{simple function} with respect to
the Hausdorff content $\mathcal{H}_{\infty}^{\delta}$ if there exist a positive integer $N$, pairwise disjoint subsets $E_1,E_2,\dots,E_N\subseteq E$ with $\mathcal{H}_{\infty}^{\delta}(E_k)<\infty$ for each $1\le k\le N$, and real constants $a_1,a_2,\dots,a_N$ such that
\[
f=\sum_{k=1}^{N}a_k\mathbf{1}_{E_k}.
\]
We denote by $\mathcal{S}(E,\mathcal{H}_{\infty}^{\delta})$ the collection of all such simple functions.

\section{Hausdorff Content and Decreasing Rearrangement}\label{s2}

In this section, we recall some notation related to the Hausdorff content and give some properties
of decreasing rearrangement.
Let $\dz \in (0,n]$. For any subset $E\subset \rn$,
the \emph{Hausdorff content} of $E$ is defined by setting
\begin{align*}
\mathcal{H}_{\infty }^{\delta }\lf (E \r )
:=\inf\lf \{\sum_i [l(Q_i)]^{\delta }:\
E\subset \bigcup_i Q_i \right \},
\end{align*}
where the infimum is taken over all finite or countable cubes coverings
$\{Q_i\}_{i}$ of $E$ and $l(Q)$ denotes the edge length
of the cube $Q$. The quantity $\mathcal{H}_{\infty }^{\delta }\lf (E \r )$ is also referred to as the $\delta$-\emph{Hausdorff content}
or the $\delta$-\emph{Hausdorff capacity} or the \emph{Hausdorff content of $E$ of
dimension} $\delta$.

For any subset $E\subset \rn$ and any non-negative function $f$, the \emph{Choquet integral} of $f$ on $E$ is defined by setting
\[\int_{E}f(x)\,d\ch_{\fz}^{\dz}\coloneqq\int_{0}^{\fz}\ch_{\fz}^{\dz}(\{x\in E:\ f(x)>t\})\,dt.\]
From the definition, it follows that the Choquet integral is well defined for all non-negative functions, without the assumption of measurability. For any $p\in(0,\fz)$ and any subset $E\subset \rn$, the \emph{$p$-Choquet
integral with respect to the capacity} of a function $f$ on $E$ is then defined by setting
\begin{align*}
\|f\|_{L^p(E,\ch_{\fz}^{\dz})}
\coloneqq\lf[\int_E |f(x)|^p\,d\ch_{\fz}^{\dz}\r]^\frac1p,
\end{align*}
and we use $L^p(E,\ch_{\fz}^{\dz})$ to denote
the space of all functions $f$ such that $\|f\|_{L^p(E,\ch_{\fz}^{\dz})}$ is finite.

Denote by $L_{\loc}^1(\rn,\ch_{\fz}^{\dz})$ the set of all functions $f$ satisfying that $f\in L^1(K,\ch_{\fz}^{\dz})$ for any compact set $K\subset\rn$.
We also use $L^{\fz}(E, \ch_\fz^\delta)$ to denote the space of all functions $f$ on $E$ such that
$$\|f\|_{L^\infty(E,\ch_\fz^\delta)}:=\inf\lf\{t>0:\ch_\fz^\delta(\{x\in E:\ |f(x)|>t\})=0\r\}<\fz.$$
It is worth noting that the Hausdorff content, originating from fractal geometry (see \cite{eg15,fe69}), plays a crucial role in harmonic analysis and nonlinear potential theory (see \cite{AdH96}). For the further theory of Choquet integral with respect to Hausdorff content, we refer the reader to the original work of G. Choquet \cite{ch55} and the excellent survey by D. R. Adams \cite{0923-2}. 

It is known that the Choquet integral with respect to $\ch_{\fz}^{\dz}$ satisfies the quasi-sublinear, namely,
$$\int_E[f(x)+g(x)]\,d\ch_{\fz}^{\delta}
\le 2\lf[\int_E f(x)\,d\ch_{\fz}^{\delta}+\int_E g(x)\,d\ch_{\fz}^{\delta}\r],$$
but it is not linear, even not sublinear.
Alternatively, we recall the \emph{dyadic Hausdorff content} $\widetilde{\ch}_\fz^{\delta}$,
which is defined by the same way as
$\ch^{\delta}_{\fz}$ but with cubes coverings $\{Q_i\}_i$ replaced by dyadic cubes coverings.
This dyadic Hausdorff content $\widetilde{\ch}^{\delta}_{\fz}$ is equivalent to $\ch_\fz^{\delta}$ (see \cite{Ad88}), namely,
there exists a positive constant $K=K(n, \delta)$ such that, for any subset $E\subset \rn$,
\begin{equation}\label{eq0723a}
\ch^{\delta}_{\fz}(E)\le \widetilde{\ch}_\fz^{\delta}(E)\le K\ch^{\delta}_{\fz}(E).
\end{equation}
Moreover, this dyadic Hausdorff content $\widetilde{\ch}^{\delta}_{\fz}$ satisfies the following properties:
If $\{E_j\}_{j\in\nn}$ is an  increasing sequence of arbitrary sets with limit $E$ (for short, $E_{j}\uparrow E$), then
\begin{equation}\label{eq0723b}
\widetilde{\ch}_\fz^\delta(E)=\lim_{j\to \fz}\widetilde{\ch}_\fz^\delta(E_j).
\end{equation}
This importantly property further leads to the sublinear of the Choquet integral with respect to the dyadic Hausdorff content
(see \cite[p.\,13, Theorem 1]{0923-2}):
for any non-negative functions $f$ and $g$ defined on $E$,
\begin{equation}\label{eq0904}
\int_{E}\lf[f(x)+g(x)\r]\,d\widetilde{\ch}_\fz^\delta
\le \int_{E}f(x)\,d\widetilde{\ch}_\fz^\delta+\int_{E}f(x)\,d\widetilde{\ch}_\fz^\delta.
\end{equation}
For the needs of the subsequent parts of this paper, we summarize some more relevant properties related to the Hausdorff content as follows.

\begin{rem}\label{12-4}
Let $\delta\in(0,n]$.
\begin{enumerate}
\item[\rm{(i)}]
Obviously, for any cube $Q$ of $\rn$, $\ch^{\delta}_{\fz}(Q)=[l(Q)]^{\delta}$ and hence
$\ch^{\delta}_{\fz}(tQ)=t^{\dz}\ch^{\delta}_{\fz}(Q)$ for any $t\in (0,\fz)$,
where $tQ$ denotes the cube that is concentric with $Q$ and has side length $tl(Q)$.

\item[\rm{(ii)}]
Let $E\subset \rn$ satisfy $\widetilde{\ch}_\fz^\delta(E)<\fz$. Then, for any function $f$ on $E$
and any positive constant $A$, we have
\[\int_{E}\lf[|f(x)|+A\r]\,d\widetilde{\ch}_\fz^\delta=\int_{E}|f(x)|\,d\widetilde{\ch}_\fz^\delta+A\widetilde{\ch}_\fz^\delta(E).\]
 Also, it is straightforward to verify that for $\ch_{\fz}^{\dz}$
\[\int_{E}[|f(x)|+A]\,d\ch_\fz^\delta=\int_{E}|f(x)|\,d\ch_\fz^\delta+A\ch_\fz^\delta(E).\]

\item[\rm{(iii)}]
By the subadditivity of $\ch_\fz^\delta$ and the definition of the Choquet integral,
we find that, for any sequence $\{E_j\}_{j\in\nn}$ of subsets in $\rn$,
$$\int_{\bigcup_{j\in\nn}E_j}|f(x)|\,d\ch_\fz^\delta\le \sum_{j\in\nn}\int_{E_j}|f(x)|\,\ch_\fz^\delta.$$
\end{enumerate}
\end{rem}

Let $\delta\in(0,n]$ and $E\subset\mathbb{R}^n$. For any function $f$ defined on $E$,
we define its \emph{decreasing rearrangement} with respect to the Hausdorff content $\mathcal{H}_{\infty}^{\delta}$, denoted by $f^{*,E}_{\mathcal{H}_{\infty}^{\delta}}$, as follows:
\begin{align*}
f^{*,E}_{\mathcal{H}_{\infty}^{\delta}}(t)
:=\inf\big\{\lambda>0 : \mathcal{H}_{\infty}^{\delta}\big(\{x\in E : |f(x)|>\lambda\}\big)\leq t\big\},
\quad \forall  t\in[0,\infty),
\end{align*}
which is also equal to $\inf\big\{\lambda\geq0 : \mathcal{H}_{\infty}^{\delta}\big(\{x\in E : |f(x)|>\lambda\}\big)\leq t\big\}$.
The \emph{second-order decreasing rearrangement} of $f$ with respect to $\mathcal{H}_{\infty}^{\delta}$ is denoted by $f^{**,E}_{\mathcal{H}_{\infty}^{\delta}}$ and defined by
\begin{equation}\label{eq0902}
f^{**,E}_{\mathcal{H}_{\infty}^{\delta}}(t)
:=\frac{1}{t}\int_0^t f^{*,E}_{\mathcal{H}_{\infty}^{\delta}}(s)\,ds,\qquad \forall\,t\in(0,\infty).
\end{equation}

Similarly, we define the decreasing rearrangement and the second-order decreasing rearrangement of $f$ with respect to $\widetilde{\mathcal{H}}_{\infty}^{\delta}$, and denote by $f^{*,E}_{\widetilde{\mathcal{H}}_{\infty}^{\delta}}$ and $f^{**,E}_{\widetilde{\mathcal{H}}_{\infty}^{\delta}}$, respectively.
In particular, if $E=\mathbb{R}^n$, then we omit the superscript $E$ for brevity and simply write
$f^{*}_{\mathcal{H}_{\infty}^{\delta}}, f^{**}_{\mathcal{H}_{\infty}^{\delta}},
f^{*}_{\widetilde{\mathcal{H}}_{\infty}^{\delta}}, f^{**}_{\widetilde{\mathcal{H}}_{\infty}^{\delta}}$instead of $f^{*,{\mathbb R}^n}_{\mathcal{H}_{\infty}^{\delta}}$, $f^{**,{\mathbb R}^n}_{\mathcal{H}_{\infty}^{\delta}}$, $f^{*,{\mathbb R}^n}_{\widetilde{\mathcal{H}}_{\infty}^{\delta}}$ and $f^{**,{\mathbb R}^n}_{\widetilde{\mathcal{H}}_{\infty}^{\delta}}$.

\begin{defn}\label{def-w}
Let $\delta\in(0, n]$ and $E$ be a subset of $\mathbb{R}^n$.
Then the \emph{weak-$L^\infty$ space with respect to the Hausdorff content} $\ch_\fz^\delta$ on $E$,
denoted by $wL^\infty(E, \mathcal{H}^{\delta}_{\infty})$,
is defined to be the set of all functions $f$ on $E$ such that
$$\|f\|_{wL^{\infty}(E, \mathcal{H}^{\delta}_{\infty})} := \sup_{t\in(0,\infty)}\left(f^{**, E}_{\mathcal{H}^{\delta}_{\infty}}(t)- f_{\mathcal{H}^{\delta}_{\infty}}^{*, E}(t)\right) < \infty.$$
\end{defn}

The following proposition presents several properties of the decreasing rearrangement in the context of the Hausdorff content.
Although their proofs are similar to that of the classical case, we give some details here for the completeness.

\begin{prop}\label{2.2}
Let $\dz\in (0,n]$, $E$ be a subset of $\rn$ and $f,\,g$ be functions defined on $E$.
\begin{enumerate}
\item[\rm{(i)}]
For any $t\in (0,\fz)$,
\[\ch_{\fz}^{\dz}\lf(\lf\{x\in E: |f(x)|>f^{*, E}_{\ch_{\fz}^{\dz}}(t)\r\}\r)\leq Kt,\]
here and thereafter, the constant $K$ is as in \eqref{eq0723a}.

\item[\rm{(ii)}]
For any $\alpha\in (0,\fz)$,
\[f^{*, E}_{\ch_{\fz}^{\dz}}\lf(\ch_{\fz}^{\dz}(\{x\in E: |f(x)|>\alpha\})\r)\leq\alpha.\]

\item[\rm{(iii)}]
For any $s, t\in [0,\fz)$, if $f^{*, E}_{\ch_{\fz}^{\dz}}(t)>s$, then $\ch_{\fz}^{\dz}(\{x\in E: |f(x)|>s\})>t$.

\item[\rm{(iv)}]
If $|g(x)|\leq|f(x)|$ for $\ch_{\fz}^{\dz}$-almost every $x\in E$,
then, for any $t\in [0, \fz)$, $g^{*, E}_{\ch_{\fz}^{\dz}}(t)\leq f^{*, E}_{\ch_{\fz}^{\dz}}(t)$.

\item[\rm{ (v)}]
 Let $\{f_{k}(x)\}_{k\in\nn}$ be a sequence of functions on $E$ satisfying that, for any $x\in E$ and any $k\in\nn$, $|f_k(x)|\leq |f_{k+1}(x)|$ and, for $\ch_{\fz}^{\dz}$-almost every $x\in E$, $\lim_{k\to\fz}|f_k(x)|=|f(x)|$.
  Then, we have
\[f^{*, E}_{\ch_{\fz}^{\dz}}(Kt)\leq \lim_{k\to\fz}(f_k)_{\ch_{\fz}^{\dz}}^{*, E}(t)\leq f^{*, E}_{\ch_{\fz}^{\dz}}(t),
\quad\forall\,t\in[0,\fz).\]

\item[\rm{(vi)}]
For any $\lambda\in [0,\fz)$, we have
\begin{align*}
\frac{1}{K}\ch_{\fz}^{\dz}(\{x\in E: |f(x)|>\lambda\})\leq\lf|\lf\{t\in [0,\infty):f^{*, E}_{\ch_{\fz}^{\dz}}(t)>\lambda\r\}\r|\leq K\ch_{\fz}^{\dz}(\{x\in E: |f(x)|>\lambda\}).
\end{align*}

\item[\rm{(vii)}]
For any $p\in(0, \fz)$,
\begin{align*}
\frac{1}{K}\int_E |f(x)|^p\,d\ch_{\fz}^{\dz}\leq \int_{0}^{\fz}\lf[f^{*, E}_{\ch_{\fz}^{\dz}}(t)\r]^p\,dt\leq K\int_E |f(x)|^p\,d\ch_{\fz}^{\dz}.
\end{align*}

\item[\rm{(viii)}]
For any $t_1, t_2\in [0, \fz)$, we have
$$(f+g)^{*, E}_{\ch_{\fz}^{\dz}}(t_1+t_2)\leq f^{*, E}_{\ch_{\fz}^{\dz}}(t_1)+g^{*, E}_{\ch_{\fz}^{\dz}}(t_2).$$

\item[\rm{(ix)}] For any $p\in(1,\infty)$,
\begin{align*}
	\int_{0}^{\fz}\lf[f^{*, E}_{\ch_{\fz}^{\dz}}(t)\r]^p\,dt\leq \int_{0}^{\fz}\lf[f^{**, E}_{\ch_{\fz}^{\dz}}(t)\r]^p\,dt\leq \lf(\frac{p}{p-1}\r)^p \int_{0}^{\fz}\lf[f^{*, E}_{\ch_{\fz}^{\dz}}(t)\r]^p\,dt.
\end{align*}
\end{enumerate}
\end{prop}

\begin{proof}
\rm(i)
For any $t\in[0,\fz)$, let $E^t:=\{s>0: \ch_{\fz}^{\dz}(\{x\in E: |f(x)|>s\})\leq t\}$, and  $\{s_k\}_{k\in\nn}\subset E^t$ be a decreasing sequence such that $\lim_{k\in\nn}s_k=f^{*, E}_{\ch_{\fz}^{\dz}}(t)$. If we define, for any $k\in\nn$,
\begin{align*}
E_{k}:=\lf\{x\in E:|f(x)|>s_{k}\r\},
\end{align*}
then $$\bigcup_{k\in\nn}E_k=\lf\{x\in E:|f(x)|>f^{*, E}_{\ch_{\fz}^{\dz}}(t)\r\}.$$
Therefore, by \eqref{eq0723a} and \eqref{eq0723b}, we find that
\begin{align}\label{eq1125-1}
\ch_{\fz}^{\dz}\lf(\lf\{x\in E: |f(x)|>f^{*, E}_{\ch_{\fz}^{\dz}}(t)\r\}\r)
&\leq \widetilde{\ch}_\fz^{\delta}\lf(\lf\{x\in E: |f(x)|>f^{*, E}_{\ch_{\fz}^{\dz}}(t)\r\}\r)=\lim_{k\to\fz}\widetilde{\ch}_\fz^{\delta}(E_k)\\
&\leq K\dlimsup_{k\to\fz}\ch_{\fz}^{\delta}(E_k)\le Kt\noz.
\end{align}

\rm{ (ii)}
For any given $\alpha\in (0, \fz)$, note that
$$\alpha\in \lf\{\lambda>0:\ch_{\fz}^{\dz}(\{x\in E: |f(x)|>\lambda\})\leq \ch_{\fz}^{\dz}(\{x\in E: |f(x)|>\alpha\})\r\}=:F.$$
Then we obtain
\begin{align*}
f^{*, E}_{\ch_{\fz}^{\dz}}\lf(\ch_{\fz}^{\dz}(\{x\in E: |f(x)|>\alpha\})\r)=\inf F\leq \alpha.
\end{align*}

\rm{(iii)}
From
\[f^{*, E}_{\ch_{\fz}^{\dz}}(t)=\inf\lf\{\lambda>0:\ch_{\fz}^{\dz}(\lf\{x\in E: |f(x)|>\lambda\r\})\leq t\r\}>s,\]
we deduce that $s\notin \{\lambda>0:\ch_{\fz}^{\dz}(\lf\{x\in E: |f(x)|>\lambda\r\})\leq t\}$.
Thus, $\ch_{\fz}^{\dz}(\{x\in E: |f(x)|>s\})>t.$

\rm{(iv)}
Since $|g(x)|\leq|f(x)|$ for $\ch_{\fz}^{\dz}$-almost every $x\in E$, it follows that, for any $\lambda\in(0,\fz)$,
\[\ch_{\fz}^{\dz}(\{x\in E: |g(x)|>\lambda\})\leq \ch_{\fz}^{\dz}(\{x\in E: |f(x)|>\lambda\}).\]
Then, for any $t\in[0,\fz)$,
\[\lf\{\lambda>0: \ch_{\fz}^{\dz}(\{x\in E: |f(x)|>\lambda\})\leq t\r\}
\subset \lf\{\lambda>0: \ch_{\fz}^{\dz}(\{x\in E: |g(x)|>\lambda\})\leq t\r\},\]
which implies that (iv) holds.

\rm(v)
Without loss of generality, we may assume that $\{f_{k}(x)\}_{k\in\nn}$ is a sequence of non-negative functions on $E$. By (iv), we find that, for any $k\in\nn$ and any $t\in[0,\fz)$,
\begin{align}\label{eq1125-2}
(f_{k})^{*, E}_{\ch_{\fz}^{\dz}}(t)\leq (f_{k+1})^{*, E}_{\ch_{\fz}^{\dz}}(t)\leq f^{*, E}_{\ch_{\fz}^{\dz}}(t).
\end{align}
Then, for a given $t\in[0,\fz)$, $\{(f_{k})^{*, E}_{\ch_{\fz}^{\dz}}(t)\}_{k\in\nn}$ converges and
$$h(t):=\lim_{k\to\fz}(f_{k})^{*, E}_{\ch_{\fz}^{\dz}}(t)\leq f^{*, E}_{\ch_{\fz}^{\dz}}(t).$$
 To complete the proof of (v), it remains to show that $f^{*, E}_{\ch_{\fz}^{\dz}}(Kt)\leq h(t)$. In fact, note that for any $k\in \nn$, $(f_{k})^{*, E}_{\ch_{\fz}^{\dz}}(t)\leq h(t)$, then by \eqref{eq0723a}, \eqref{eq0723b}, \eqref{eq1125-2} and \eqref{eq1125-1}, we have
\begin{align*}
\ch_{\fz}^{\dz}\lf(\lf\{x\in E: |f(x)|>h(t)\r\}\r)&\leq \widetilde{\ch}_\fz^{\delta}\lf(\lf\{x\in E: |f(x)|>h(t)\r\}\r)\\
&=\lim_{k\to\fz}\widetilde{\ch}_\fz^{\delta}\lf(\lf\{x\in E: |f_k(x)|>h(t)\r\}\r)\\
&\leq \limsup_{k\to\fz}\widetilde{\ch}_\fz^{\delta}\lf(\lf\{x\in E: |f_k(x)|>(f_{k})^{*, E}_{\ch_{\fz}^{\dz}}(t)\r\}\r)\\
&\leq Kt.
\end{align*}
Subsequently, by applying (iii), we immediately obtain $f^{*, E}_{\ch_{\fz}^{\dz}}(Kt)\leq h(t)$.

\rm{(vi)}
We first show that for any non-negative simple function $f$ on $E$ with respect to $\ch_\fz^\delta$,
\begin{align}\label{eq1125-3}
\ch_{\fz}^{\dz}\lf(\lf\{x\in E: |f(x)|>\lambda\r\}\r)=\lf|\lf\{t\in [0,\fz): f^{*, E}_{\ch_{\fz}^{\dz}}(t)>\lambda\r\}\r|,
\quad \forall\,\lambda\in [0,\fz).
\end{align}
Without loss of generality, we may assume $f:=\sum_{j=1}^{N}a_j\mathbf{1}_{E_j}$ with some $N\in\nn$,
where $\{E_j\}_{j=1}^N$ are pairwise disjoint subsets of $E$ and $a_{1}>a_{2}>\cdots>a_{N}>0$.
Additionally, we set $a_0=\fz$ and $a_{N+1}=0$. Then, when $\lambda\in [a_k, a_{k-1})$ with $k=2,\cdots, N+1$,
\begin{align}\label{eq1125-4}
\ch_{\fz}^{\dz}(\{x\in E: |f(x)|>\lambda\})=\ch_{\fz}^{\dz}\lf(\bigcup_{j=1}^{k-1}E_j\r),
\end{align}
and when $\lambda\in [a_1, a_{0})$,
\begin{align}\label{eq1125-5}
\ch_{\fz}^{\dz}(\{x\in E: |f(x)|>\lambda\})=\ch_{\fz}^{\dz}(\emptyset)=0.
\end{align}
For any $k=1,\cdots, N$, let $b_{k}:=\mathcal{H}_{\infty}^{\delta}\big(\bigcup_{j=1}^{k}E_j\big)$ and set $b_{0}=0$.
Then we easily conclude that for any $t\in [0,\infty)$,
\begin{align}\label{eq521-8}
f^{*, E}_{\mathcal{H}_{\infty}^{\delta}}(t)=\sum_{k=1}^{N}a_k\mathbf{1}_{[b_{k-1},b_k)}(t).
\end{align}
Here, we remark that if $b_{k-1}=b_{k}$ holds for some index $k$,
then the corresponding term $a_k\mathbf{1}_{[b_{k-1},b_k)}$ is omitted in the summation.
Subsequently, when $\lambda\in [a_k, a_{k-1})$ with $k=2,\cdots, N+1$,
\begin{align*}
\lf|\lf\{t\in [0,\fz): f^{*, E}_{\ch_{\fz}^{\dz}}(t)>\lambda\r\}\r|=|[0, b_{k-1})|=\ch_{\fz}^{\dz}\lf(\bigcup_{j=1}^{k-1}E_j\r),
\end{align*}
and when $\lambda\in [a_1, a_{0})$,
\[\lf|\lf\{t\in [0,\fz): f^{*, E}_{\ch_{\fz}^{\dz}}(t)>\lambda\r\}\r|=|\emptyset|=0.\]
Combining this, \eqref{eq1125-4} and \eqref{eq1125-5}, we immediately obtain \eqref{eq1125-3}.

Now, for any function $f$ defined on $E$, we choose an increasing sequence of non-negative simple functions
$\{f_k\}_{k\in\mathbb N}$ with respect to $\ch_\fz^\delta$ satisfying $f_k(x)\le f_{k+1}(x)$
and $ \lim_{k\to\infty}f_k(x)=|f(x)|$ for $\ch_\fz^\delta$-almost every $x\in E$.
Applying (v) of this proposition, we have
\begin{align}\label{eq1126-1}
f^{*, E}_{\ch_{\fz}^{\dz}}(Kt)\leq \lim_{k\to\fz}(f_k)_{\ch_{\fz}^{\dz}}^{*, E}(t)\leq f^{*, E}_{\ch_{\fz}^{\dz}}(t), \quad\forall t\in [0,\fz),
\end{align}
and, by (iv),
\begin{align*}
(f_k)_{\ch_{\fz}^{\dz}}^{*, E}(t)\leq (f_{k+1})_{\ch_{\fz}^{\dz}}^{*, E}(t)
\leq \lim_{k\to\fz}(f_k)_{\ch_{\fz}^{\dz}}^{*, E}(t) \quad\forall t\in [0,\fz),\ \forall k\in\nn.
\end{align*}
From this, together with \eqref{eq0723a}, \eqref{eq0723b} and \eqref{eq1125-3}, we infer that, for any $\lz\in(0,\fz)$,
\begin{align}\label{eq1126-3}
\ch_{\fz}^{\dz}(\{x\in E: |f(x)|>\lambda\})
&\leq \widetilde{\ch}_\fz^{\delta}(\{x\in E: |f(x)|>\lambda\})\\
&=\lim_{k\to\fz}\widetilde{\ch}_\fz^{\delta}(\{x\in E: f_k(x)>\lambda\})\noz\\
&\leq K\limsup_{k\to\fz}{\ch}_\fz^{\delta}(\{x\in E: f_k(x)>\lambda\})\noz\\
&=K\limsup_{k\to\fz}\lf|\lf\{t\in [0,\fz): (f_k)^{*, E}_{\ch_{\fz}^{\dz}}(t)>\lambda\r\}\r|\noz\\
&\leq K\lf|\lf\{t\in [0,\fz):\lim_{k\to\fz} (f_k)^{*, E}_{\ch_{\fz}^{\dz}}(t)>\lambda\r\}\r|\noz\\
&\leq K\lf|\lf\{t\in [0,\fz):f^{*, E}_{\ch_{\fz}^{\dz}}(t)>\lambda\r\}\r|\noz.
\end{align}
Conversely, by \eqref{eq1125-3} and \eqref{eq1126-1}, we find that, for any $\lz\in(0,\fz)$,
\begin{align*}
\lf|\lf\{t\in [0,\fz):f^{*, E}_{\ch_{\fz}^{\dz}}(t)>\lambda\r\}\r|
&=K\lf|\lf\{K^{-1}t\in [0,\fz):f^{*, E}_{\ch_{\fz}^{\dz}}(t)>\lambda\r\}\r|\\
&=K\lf|\lf\{t\in [0,\fz):f^{*, E}_{\ch_{\fz}^{\dz}}(Kt)>\lambda\r\}\r|\\
&\leq K\lf|\lf\{t\in [0,\fz):\lim_{k\to\fz}(f_k)_{\ch_{\fz}^{\dz}}^{*, E}(t)>\lambda\r\}\r|\\
&=K\lim_{k\to\fz}\lf|\lf\{t\in [0,\fz):(f_k)_{\ch_{\fz}^{\dz}}^{*, E}(t)>\lambda\r\}\r|\\
&=K\lim_{k\to\fz}{\ch}_\fz^{\delta}(\{x\in E: f_k(x)>\lambda\})\\
&\leq K{\ch}_\fz^{\delta}(\{x\in E: |f(x)|>\lambda\}).
\end{align*}
Therefore, (vi) holds true.

\rm{(vii)} is a consequence of (vi).

\rm{(viii)}
For any $t_1, t_1\in [0,\fz)$, we let
$$A:=\lf\{\lambda>0: \ch_{\fz}^{\dz}(\{x\in E: |f(x)|>\lambda\})\leq t_1\r\}\ {\rm{and}}\ B:=\lf\{\lambda>0: \ch_{\fz}^{\dz}(\{x\in E: |g(x)|>\lambda\})\leq t_2\r\}.$$
Then, for any $s_1\in A$ and any $s_2\in B$,
\[\ch_{\fz}^{\dz}(\{x\in E: |f(x)|>s_1\})\leq t_1\quad {\rm{and}} \quad\ch_{\fz}^{\dz}(\{x\in E: |g(x)|>s_2\})\leq t_2.\]
By this, we further obtain
\[\ch_{\fz}^{\dz}(\{x\in E: |f(x)+g(x)|>s_1+s_2\})\leq\ch_{\fz}^{\dz}(\{x\in E: |f(x)|+|g(x)|>s_1+s_2\})\leq t_1+t_2,\]
which implies that $(f+g)^{*, E}_{\ch_{\fz}^{\dz}}(t_1+t_2)\leq s_1+s_2$. Thus, (viii) is true.

\rm(ix) Since $f^{*,E}_{\mathcal{H}_{\infty}^{\delta}}(t)\leq f^{**,E}_{\mathcal{H}_{\infty}^{\delta}}(t)$
for all $t\in (0,\infty)$, the first inequality follows immediately.
For the second inequality, applying the Hardy inequality, we get
\begin{align*}
\int_{0}^{\infty}\lf[f^{**,E}_{\mathcal{H}_{\infty}^{\delta}}(t)\r]^pdt
=\int_{0}^{\infty}\left(\frac{1}{t}\int_0^t f^{*,E}_{\mathcal{H}_{\infty}^{\delta}}(s)ds\right)^p\,dt
\leq \left(\frac{p}{p-1}\right)^p \int_{0}^{\infty}\lf[f^{*,E}_{\mathcal{H}_{\infty}^{\delta}}(t)\r]^p\,dt.
\end{align*}
Hence assertion (ix) is valid.
\end{proof}

\begin{rem}\label{rem59-1}
It is not difficulty to see that
all properties in Proposition \ref{2.2} remain valid with
${H}_{\infty}^{\delta}$ replaced by $\widetilde{\mathcal{H}}_{\infty}^{\delta}$,
moreover, the corresponding constant can be taken as $K=1$.
\end{rem}

In the following proposition, using a $K$-functional, we represent the decreasing rearrangement.

\begin{prop}\label{prop59-1}
Let $\dz\in(0,n]$, $E$ be a subset of $\rn$ and $f$ a function defined on $E$. Then
\begin{align*}
\inf_{f=g+h} \lf\{
\|g\|_{L^1(E, \widetilde{\mathcal{H}}_{\infty}^{\delta})}
+ t\|h\|_{L^{\infty}(E, \widetilde{\mathcal{H}}_{\infty}^{\delta})}
\r\}=t f_{\widetilde{\mathcal{H}}_{\infty}^{\delta}}^{**, E}(t),\quad \forall\,t\in(0,\fz),
\end{align*}
where the infimum is taken over all decomposition of
$f=g+h\in L^1(E,\widetilde{\mathcal{H}}_{\infty}^{\delta}))+L^\fz(E,\widetilde{\mathcal{H}}_{\infty}^{\delta}))$.
\end{prop}

\begin{proof}
For any fixed $t\in(0,\infty)$, we first establish
\begin{align}\label{eq592}
t f_{\widetilde{\mathcal{H}}_{\infty}^{\delta}}^{**, E}(t)
\leq \inf_{f=g+h} \lf\{
\|g\|_{L^1(E, \widetilde{\mathcal{H}}_{\infty}^{\delta})}
+ t\|h\|_{L^{\infty}(E, \widetilde{\mathcal{H}}_{\infty}^{\delta})}
\r\}.
\end{align}
We may assume without loss of generality that
\[
\inf_{f=g+h} \lf\{
\|g\|_{L^1(E, \widetilde{\mathcal{H}}_{\infty}^{\delta})}
+ t\|h\|_{L^{\infty}(E, \widetilde{\mathcal{H}}_{\infty}^{\delta})}
\r\}<\infty.
\]
Then, for any $\varepsilon\in(0,\fz)$, there exist functions $g_1\in L^1(E, \widetilde{\mathcal{H}}_{\infty}^{\delta})$
and $h_1\in L^{\infty}(E, \widetilde{\mathcal{H}}_{\infty}^{\delta})$ on $E$ such that $f=g_1+h_1$ and
\begin{align}\label{eq593}
\|g_1\|_{L^1(E, \widetilde{\mathcal{H}}_{\infty}^{\delta})}
+ t\|h_1\|_{L^{\infty}(E, \widetilde{\mathcal{H}}_{\infty}^{\delta})}
\leq
\inf_{f=g+h} \lf\{
\|g\|_{L^1(E, \widetilde{\mathcal{H}}_{\infty}^{\delta})}
+ t\|h\|_{L^{\infty}(E, \widetilde{\mathcal{H}}_{\infty}^{\delta})}
\r\}+\varepsilon.
\end{align}
By (vii) and (viii) of Proposition \ref{2.2}, Remark \ref{rem59-1} and using the fact that
\[
(h_1)_{\widetilde{\mathcal{H}}_{\infty}^{\delta}}^{*, E}(0)=\|h_1\|_{L^{\infty}(E, \widetilde{\mathcal{H}}_{\infty}^{\delta})},
\]
we find that, for any $\eta\in(0,1)$,
\begin{align*}
tf_{\widetilde{\mathcal{H}}_{\infty}^{\delta}}^{**, E}(t)
&=\int_{0}^t f_{\widetilde{\mathcal{H}}_{\infty}^{\delta}}^{*, E}(s) \,ds
=\int_{0}^t \lf(g_1+h_1\r)_{\widetilde{\mathcal{H}}_{\infty}^{\delta}}^{*, E}(s) \,ds\\
&\leq \int_{0}^t (g_1)_{\widetilde{\mathcal{H}}_{\infty}^{\delta}}^{*, E} \lf(\eta s\r) \,ds
+ \int_{0}^t (h_1)_{\widetilde{\mathcal{H}}_{\infty}^{\delta}}^{*, E} \lf(s-\eta s\r) \,ds\\
& = \frac{1}{\eta}\int_{0}^{\eta t} (g_1)_{\widetilde{\mathcal{H}}_{\infty}^{\delta}}^{*, E}(u)\,du
+ \frac{1}{1-\eta}\int_{0}^{(1-\eta)t} (h_1)_{\widetilde{\mathcal{H}}_{\infty}^{\delta}}^{*, E}(u) \,du\\
& \leq \frac{1}{\eta} \int_{0}^{\infty} (g_1)_{\widetilde{\mathcal{H}}_{\infty}^{\delta}}^{*, E}(u) \,du
+\frac{1}{1-\eta}(h_1)_{\widetilde{\mathcal{H}}_{\infty}^{\delta}}^{*, E}(0) \int_{0}^{(1-\eta)t} 1 \,du\\
& =\frac{1}{\eta} \|g_1\|_{L^1(E, \widetilde{\mathcal{H}}_{\infty}^{\delta})}
+ t\|h_1\|_{L^{\infty}(E, \widetilde{\mathcal{H}}_{\infty}^{\delta})}.
\end{align*}
Letting $\eta\to 1^-$ and using \eqref{eq593} together with the arbitrariness of $\varepsilon$,
 \eqref{eq592} follows immediately.

We next prove the converse part, namely, for any $t\in(0,\fz)$,
\begin{align}\label{eq594}
\inf_{f=g+h} \lf\{
\|g\|_{L^1(E, \widetilde{\mathcal{H}}_{\infty}^{\delta})}
+ t\|h\|_{L^{\infty}(E, \widetilde{\mathcal{H}}_{\infty}^{\delta})}
\r\}
\leq t f_{\widetilde{\mathcal{H}}_{\infty}^{\delta}}^{**, E}(t)
=\int_0^t f_{\widetilde{\mathcal{H}}_{\infty}^{\delta}}^{*, E}(s)\,ds.
\end{align}
Without loss of generality, we may assume that $f_{\widetilde{\mathcal{H}}_{\infty}^{\delta}}^{**, E}(t)<\infty$.
Obviously,
$
f_{\widetilde{\mathcal{H}}_{\infty}^{\delta}}^{*, E}(t)\leq f_{\widetilde{\mathcal{H}}_{\infty}^{\delta}}^{**, E}(t).
$
Define
\[
F_t:=\lf\{x\in E: |f(x)|>f_{\widetilde{\mathcal{H}}_{\infty}^{\delta}}^{*, E}(t)\r\},\quad \forall\,t\in(0,\fz).
\]
Then $\widetilde{\mathcal{H}}_{\infty}^{\delta}(F_t)\le t$ by (i) of Proposition \ref{2.2}.
Given $t\in(0,\fz)$, let, for any $x\in E$,
\[
g_2(x) :=\lf[|f(x)|-f_{\widetilde{\mathcal{H}}_{\infty}^{\delta}}^{*, E}(t) \r]\mathbf{1}_{F_t}(x){\rm sgn} f(x),
\qquad
h_2(x) := \min\lf\{|f(x)|,f_{\widetilde{\mathcal{H}}_{\infty}^{\delta}}^{*, E}(t) \r\}{\rm sgn} f(x).
\]
Clearly, $f=g_2+h_2$. Since \(|h_2(x)|\leq f_{\widetilde{\mathcal{H}}_{\infty}^{\delta}}^{*, E}(t)\)
for all \(x\in E\), it follows from Remark \ref{12-4}(ii) that
\begin{align}\label{eq510-a}
\inf_{f=g+h} \lf\{
\|g\|_{L^1(E, \widetilde{\mathcal{H}}_{\infty}^{\delta})}
+ t\|h\|_{L^{\infty}(E, \widetilde{\mathcal{H}}_{\infty}^{\delta})}
\r\}
&\leq \lf\{
\|g_2\|_{L^1(E, \widetilde{\mathcal{H}}_{\infty}^{\delta})}
+ t\|h_2\|_{L^{\infty}(E, \widetilde{\mathcal{H}}_{\infty}^{\delta})}
\r\}\\
&\leq \int_F \lf(|f(x)|-f_{\widetilde{\mathcal{H}}_{\infty}^{\delta}}^{*,E}(t)\r)\,d \widetilde{\mathcal{H}}_{\infty}^{\delta}
+ t f_{\widetilde{\mathcal{H}}_{\infty}^{\delta}}^{*, E}(t)\notag\\
&=\int_E |(f\mathbf{1}_{F_t})(x)|\,d \widetilde{\mathcal{H}}_{\infty}^{\delta}
+ \lf(t-\widetilde{\mathcal{H}}_{\infty}^{\delta}(F_t)\r) f_{\widetilde{\mathcal{H}}_{\infty}^{\delta}}^{*, E}(t).\notag
\end{align}

When  \(s\ge \widetilde{\mathcal{H}}_{\infty}^{\delta}(F_t)\), it holds true that
\(\widetilde{\mathcal{H}}_{\infty}^{\delta}(\{x\in E: |(f\mathbf{1}_{F_t})(x)|>\lambda\})\le s\)
for any \(\lambda\in(0,\fz)\), and hence
\((f\mathbf{1}_{F_t})^{*,E}_{\widetilde{\mathcal{H}}_{\infty}^{\delta}}(s)=0.\)
This, combined with  (iv) and (vii) of Proposition \ref{2.2} and Remark \ref{rem59-1}, infers that
\begin{align}\label{eq510-1}
\int_E |(f\mathbf{1}_{F_t})(x)|\,d\widetilde{\mathcal{H}}_{\infty}^{\delta}
=\int_0^{\widetilde{\mathcal{H}}_{\infty}^{\delta}(F_t)} (f\mathbf{1}_{F_t})^{*,E}_{\widetilde{\mathcal{H}}_{\infty}^{\delta}}(s)\,ds
\leq\int_0^{\widetilde{\mathcal{H}}_{\infty}^{\delta}(F_t)} f^{*,E}_{\widetilde{\mathcal{H}}_{\infty}^{\delta}}(s)\,ds.
\end{align}
If $\widetilde{\mathcal{H}}_{\infty}^{\delta}(F_t)=t$, then \eqref{eq594} follows directly
from \eqref{eq510-a} and \eqref{eq510-1}.
If $\widetilde{\mathcal{H}}_{\infty}^{\delta}(F_t)<t$, then
for any $s\in(\widetilde{\mathcal{H}}_{\infty}^{\delta}(F_t),t)$,
\[
\widetilde{\mathcal{H}}_{\infty}^{\delta}\lf(\lf\{x\in E: |f(x)|>f^{*,E}_{\widetilde{\mathcal{H}}_{\infty}^{\delta}}(t)\r\}\r)
=\widetilde{\mathcal{H}}_{\infty}^{\delta}(F_t)<s,
\]
and hence $f^{*,E}_{\widetilde{\mathcal{H}}_{\infty}^{\delta}}(s)\leq f^{*,E}_{\widetilde{\mathcal{H}}_{\infty}^{\delta}}(t)$,
which, together with the decreasing property of $f^{*,E}_{\widetilde{\mathcal{H}}_{\infty}^{\delta}}$, implies that
\[
f^{*,E}_{\widetilde{\mathcal{H}}_{\infty}^{\delta}}(s)=f^{*,E}_{\widetilde{\mathcal{H}}_{\infty}^{\delta}}(t),
\quad \forall\, s\in\lf(\widetilde{\mathcal{H}}_{\infty}^{\delta}(F_t),t\r).
\]
Therefore, by \eqref{eq510-a} and \eqref{eq510-1}, we conclude that \eqref{eq594} holds.
The proof of Proposition \ref{prop59-1} is then completed.
\end{proof}

As a consequence of Proposition \ref{prop59-1}, we obtain the following sublinearity of the decreasing arrangement with
respect to the dyadic Hausdorff content $\widetilde{\mathcal{H}}_{\infty}^{\delta}$. This plays a crucial role in the proof of Theorem \ref{tm0209-1}.

\begin{cor}\label{cor510-1}
Let $\delta\in(0,n]$, $E\subset\mathbb{R}^n$, and $f_1,f_2$ be functions defined on $E$. Then, for any $t\in(0,\infty)$,
\begin{align}\label{eq510-b}
(f_1+f_2)_{\widetilde{\mathcal{H}}_{\infty}^{\delta}}^{**, E}(t)
\leq (f_1)_{\widetilde{\mathcal{H}}_{\infty}^{\delta}}^{**, E}(t)
+ (f_2)_{\widetilde{\mathcal{H}}_{\infty}^{\delta}}^{**, E}(t).
\end{align}
\end{cor}

\begin{proof}
Let $t\in(0,\infty)$ be fixed.
Applying Proposition \ref{prop59-1} to $f_i$, $i\in\{1,2\}$, we know that, for any $\varepsilon\in(0,\fz)$,
there exist functions $g_i\in L^1(E, \widetilde{\mathcal{H}}_{\infty}^{\delta})$ and
$h_i\in L^{\infty}(E, \widetilde{\mathcal{H}}_{\infty}^{\delta})$ such that $f_i=g_i+h_i$ and
\[
\|g_i\|_{L^1(E, \widetilde{\mathcal{H}}_{\infty}^{\delta})}
+ t\|h_i\|_{L^{\infty}(E, \widetilde{\mathcal{H}}_{\infty}^{\delta})}
\leq t (f_i)_{\widetilde{\mathcal{H}}_{\infty}^{\delta}}^{**, E}(t)+\varepsilon.
\]
Then $f_1+f_2=(g_1+g_2)+(h_1+h_2)\in L^1(E, \widetilde{\mathcal{H}}_{\infty}^{\delta})
+L^{\infty}(E, \widetilde{\mathcal{H}}_{\infty}^{\delta})$. Thus, using Proposition \ref{prop59-1} again, we get
\[
\begin{aligned}
t (f_1+f_2)_{\widetilde{\mathcal{H}}_{\infty}^{\delta}}^{**, E}(t)
&\leq \|g_1+g_2\|_{L^1(E, \widetilde{\mathcal{H}}_{\infty}^{\delta})}
+ t\|h_1+h_2\|_{L^{\infty}(E, \widetilde{\mathcal{H}}_{\infty}^{\delta})}\\
&\leq \lf(
\|g_1\|_{L^1(E, \widetilde{\mathcal{H}}_{\infty}^{\delta})}
+ t\|h_1\|_{L^{\infty}(E, \widetilde{\mathcal{H}}_{\infty}^{\delta})}
\r)+\lf(
\|g_2\|_{L^1(E, \widetilde{\mathcal{H}}_{\infty}^{\delta})}
+ t\|h_2\|_{L^{\infty}(E, \widetilde{\mathcal{H}}_{\infty}^{\delta})}
\r)\\
&\leq t (f_1)_{\widetilde{\mathcal{H}}_{\infty}^{\delta}}^{**, E}(t)
+ t (f_2)_{\widetilde{\mathcal{H}}_{\infty}^{\delta}}^{**, E}(t)+2\varepsilon,
\end{aligned}
\]
which leads to \eqref{eq510-b} immediately.
This finishes the proof of Corollary \ref{cor510-1}.
\end{proof}

\section{Proof of Theorem \ref{tm0209-1}\label{Sec3}}

We begin with some properties of the space $L^p(\rn,\mathcal{H}_{\infty}^{\delta})$.
Recall that $\mathcal{S}(\rn,\mathcal{H}_{\infty}^{\delta})$ denotes the collection of all simple functions on $\rn$.

\begin{lem}\label{lm518-1}
Let $\delta\in(0, n]$ and $p\in(0, \infty)$. Then $\mathcal{S}(\rn,\mathcal{H}_{\infty}^{\delta})$ is dense in $L^p(\mathbb R^n,\mathcal{H}_{\infty}^{\delta})$.
\end{lem}

\begin{proof}
Let $f\in L^p(\mathbb{R}^n,\mathcal{H}_{\infty}^{\delta})$. We only need to show that, for any $\varepsilon\in(0,\fz)$, there exists a function $g\in\mathcal{S}(\rn,\mathcal{H}_{\infty}^{\delta})$ satisfying
\begin{align}\label{eq518-1}
\|f-g\|_{L^p(\mathbb{R}^n,\mathcal{H}_{\infty}^{\delta})}<\varepsilon.
\end{align}

First, we prove that there exists $N_1\in\mathbb{N}$ such that
\begin{align}\label{eq518-2}
\int_{\mathbb{R}^n\setminus B(0,2^{N_1})}|f(x)|^p\,d\mathcal{H}_{\infty}^{\delta}<\frac{\varepsilon^p}{2}.
\end{align}
It is known from \cite[Proposition 2.17]{0702-4} that the function $w(x):=|x|$, $\forall\,x\in\rn$, belongs to
the class of Hausdorff content \(A_p\) weight. From this, Proposition 2.7 and Remark 2.9 (iv) of \cite{0702-4},
we deduce that there exists a positive constant $C$ such that, for any set $E\subset\mathbb{R}^n$,
\begin{equation*}
\sum_{i\in\mathbb{Z}}\mathcal{H}_{\infty}^{\delta}(E\cap E_i)\leq C\mathcal{H}_{\infty}^{\delta}(E),
\end{equation*}
where, for any $i\in\zz$, $E_{i}:=\{x\in\mathbb{R}^n:2^{i-1}<|x|\le 2^i\}$.
This further implies that
\begin{align*}
\sum_{i\in\mathbb{Z}}\int_{E_i}|f(x)|^p\,d\mathcal{H}_{\infty}^{\delta}
\leq K\int_{\mathbb{R}^n}|f(x)|^p\,d\mathcal{H}_{\infty}^{\delta}<\infty.
\end{align*}
and hence there exists a number $N_1\in\mathbb{N}$ such that
\[
\sum_{i=N_1+1}^{\infty}\int_{E_i}|f(x)|^p\,d\mathcal{H}_{\infty}^{\delta}<\frac{\varepsilon^p}{2}.
\]
This, combining with Remark \ref{12-4}(iii), infers \eqref{eq518-2}.

For any $k\in\nn$, we define
\[
g_{k}(x):=
\begin{cases}
f(x), & x\in B(0,2^{N_1})\text{ and }|f(x)|\leq 2^k,\\
2^k, & x\in B(0,2^{N_1})\text{ and }|f(x)|>2^k,\\
0, & x\in\mathbb{R}^n\setminus B(0,2^{N_1}).
\end{cases}
\]
Next, we show that there exists $N_2\in\mathbb{N}$ such that
\begin{align}\label{eq518-4}
\int_{B(0,2^{N_1})}\big|f(x)-g_{N_2}(x)\big|^p\,d\mathcal{H}_{\infty}^{\delta}<\frac{\varepsilon^p}{2^{p+3}}.
\end{align}
Using Remark \ref{12-4}(iii) again, we obtain
\begin{align}\label{eq518-5}
\int_{B(0,2^{N_1})}\big|f(x)-g_{k}(x)\big|^p\,d\mathcal{H}_{\infty}^{\delta}
&\leq 2^p\int_{\lf\{y\in B(0,2^{N_1}):|f(y)|>2^k\r\}}|f(y)|^p\,d\mathcal{H}_{\infty}^{\delta}\\
&\leq 2^p\sum_{m=k+1}^{\infty}\int_{\lf\{y\in B(0,2^{N_1}):2^{m-1}<|f(y)|\leq 2^m\r\}}|f(y)|^p\,d\mathcal{H}_{\infty}^{\delta}.\noz
\end{align}
By an argument similar to that used in the proof of \cite[Remark 2.8]{0702-4}, we find that
\[
\sum_{m\in\mathbb{Z}}^{\infty}\int_{\lf\{y\in B(0,2^{N_1}):2^{m-1}<|f(y)|\leq 2^m\r\}}|f(y)|^p\,d\mathcal{H}_{\infty}^{\delta}
\leq \frac{2^{2p}}{2^p-1}\int_{B(0,2^{N_1})}|f(y)|^p\,d\mathcal{H}_{\infty}^{\delta}<\infty.
\]
Then we can pick $N_2\in\mathbb{N}$ such that
\[
\sum_{m=N_2+1}^{\infty}\int_{\lf\{y\in B(0,2^{N_1}):2^{m-1}<|f(y)|\leq 2^m\r\}}|f(y)|^p\,d\mathcal{H}_{\infty}^{\delta}<\frac{\varepsilon^p}{2^{2p+3}},
\]
which, combined with \eqref{eq518-5}, yields that \eqref{eq518-4} is valid.

Let, for any $x\in\mathbb{R}^n$,
\[
g_{N_2}^+(x):=\max\big\{g_{N_2}(x),0\big\},\quad
g_{N_2}^-(x):=\max\big\{-g_{N_2}(x),0\big\}.
\]
Then it is clear that $g_{N_2}=g_{N_2}^+-g_{N_2}^-$.
We let $N_3\in\mathbb{N}$ be such that
\[
\frac{1}{2^{N_3}}<\frac{\varepsilon}{2^{2+\frac{3}{p}}\mathcal{H}_{\infty}^{\delta}\big(B(0, 2^{N_1})\big)^{\frac{1}{p}}},
\]
and define, for any $x\in\rn$,
\[
g^+(x):=\sum_{j=1}^{2^{N_2+N_3}}\frac{j}{2^{N_3}}
\mathbf{1}_{\{x\in\mathbb{R}^n:\frac{j-1}{2^{N_3}}<g_{N_2}^+(x)\le\frac{j}{2^{N_3}}\}}(x),
\]
\[
g^-(x):=\sum_{j=1}^{2^{N_2+N_3}}\frac{j}{2^{N_3}}
\mathbf{1}_{\{x\in\mathbb{R}^n:\frac{j-1}{2^{N_3}}<g_{N_2}^-(x)\le\frac{j}{2^{N_3}}\}}(x),
\]
and $g(x):=g^+(x)-g^-(x)$. Then $g\in \cs(\rn)$ and, for any $x\in\mathbb{R}^n$, we have
\[
\begin{aligned}
\big|g(x)-g_{N_2}(x)\big|\le \big|g^+(x)-g_{N_2}^+(x)\big|+\big|g^-(x)-g_{N_2}^-(x)\big|\le \frac{2}{2^{N_3}}
<\frac{\varepsilon}{2^{1+\frac{3}{p}}\mathcal{H}_{\infty}^{\delta}\big(B(0, 2^{N_1})\big)^{\frac{1}{p}}},
\end{aligned}
\]
which further derives that
\begin{align}\label{eq518-6}
\int_{B(0,2^{N_1})}\big|g(x)-g_{N_2}(x)\big|^p\,d\mathcal{H}_{\infty}^{\delta}<\frac{\varepsilon^p}{2^{p+3}}.
\end{align}
Using Remark \ref{12-4}(iii), combining \eqref{eq518-2}, \eqref{eq518-4} and \eqref{eq518-6}, we conclude that
\begin{align*}
\begin{aligned}
\int_{\mathbb{R}^n}|f(x)-g(x)|^p\,d\mathcal{H}_{\infty}^{\delta}
&\leq \int_{\mathbb{R}^n\setminus B(0,2^{N_1})}|f(x)-g(x)|^p\,d\mathcal{H}_{\infty}^{\delta}
+\int_{B(0,2^{N_1})}|f(x)-g(x)|^p\,d\mathcal{H}_{\infty}^{\delta}\\
&\leq \int_{\mathbb{R}^n\setminus B(0,2^{N_1})}|f(x)|^p\,d\mathcal{H}_{\infty}^{\delta}
+2^{p+1}\int_{B(0,2^{N_1})}\big|f(x)-g_{N_2}(x)\big|^p\,d\mathcal{H}_{\infty}^{\delta}\\
&\quad +2^{p+1}\int_{B(0,2^{N_1})}\big|g_{N_2}(x)-g(x)\big|^p\,d\mathcal{H}_{\infty}^{\delta}
<\varepsilon^p.
\end{aligned}
\end{align*}
This proves \eqref{eq518-1} and hence Lemma \ref{lm518-1}.
\end{proof}

Based on \eqref{eq0723a}, \eqref{eq0723b} and an argument similar to that used in the proof of  \cite[Proposition 2.3]{OP22}, we establish the following lemma, the details being omitted.

\begin{lem}\label{lm518-2}
Let $\delta\in(0,n]$ and $p\in(0,\infty)$. Then, for any sequence $\{f_k\}_{k\in\mathbb{N}}$ of
functions in the space $L^p(\mathbb{R}^n,\mathcal{H}_{\infty}^{\delta})$ satisfying
\[
\lim_{k,m\to\infty}\|f_k-f_m\|_{L^p(\mathbb{R}^n,\mathcal{H}_{\infty}^{\delta})}=0,
\]
there exists a function $f\in L^p(\mathbb{R}^n,\mathcal{H}_{\infty}^{\delta})$ such that
\[
\lim_{k\to\infty}\|f_k-f\|_{L^p(\mathbb{R}^n,\mathcal{H}_{\infty}^{\delta})}=0.
\]
\end{lem}

As a special case of \cite[p.288. Lemma 3 ]{0304-1}, we have the following conclusion.

\begin{lem}\label{lm520-1}
Let $1\le p_1<p_2<\infty$.
Then, for any $p\in(p_1,p_2)$, there exists a positive constant $C$ depending only on $p_1,p_2$ and $p$
such that, for all non-negative non-increasing functions $f\in L^p(\mathbb{R}^+)$,
$$
\|\mathcal Rf\|_{L^p(\mathbb{R}^+)}\le C\|f\|_{L^p(\mathbb{R}^+)},
$$
where the operator $\mathcal R$ is defined by
$$
\mathcal R f(t):=\frac{1}{p_1 t^{\frac{1}{p_1}}}\int_0^t s^{\frac{1}{p_1}-1}f(s)\,ds
+\frac{1}{p_2 t^{\frac{1}{p_2}}}\int_t^\infty s^{\frac{1}{p_2}-1}f(s)\,ds, \quad \forall \ t\in(0,\fz).
$$
\end{lem}

Before proceeding to prove Theorem \ref{tm0209-1}, we recall the following operator $S$ (see \cite{BDS-1}),
which is very related to the operator $\mathcal R$. Let, for any function $f$ defined on $\mathbb R_+$,
\[
(Sf)(t):=\frac{1}{t}\int_0^t f(u)\,du + \int_t^\infty \frac{f(u)}{u}\,du,\quad \forall\,t\in(0,\infty).
\]

\begin{proof}[Proof of Theorem \ref{tm0209-1}]
We prove this theorem by four steps.

\textbf{Step i)} We show that, for any subset \(E\subset\mathbb{R}^n\) with \(0<\mathcal{H}_{\infty}^{\delta}(E)<\infty\),
\begin{align}\label{eq520-1}
(T\mathbf{1}_E)^{**}_{\mathcal{H}^\delta_\infty}(t)\leq CS\left((\mathbf{1}_E)^{**}_{\mathcal{H}^\delta_\infty}\right)(t),\quad \forall\ t\in(0,\fz).
\end{align}

For simplicity, we let \(s:=\mathcal{H}_{\infty}^{\delta}(E)\).
Then we have
\begin{align}\label{eq521-4}
(\mathbf{1}_E)^*_{\mathcal{H}_{\infty}^{\delta}}(t)=
\begin{cases}
1, & t\in\big(0,s\big),\\
0, & t\in\big[s,\infty\big).
\end{cases}
\end{align}
By this, the property of restricted weak type \((1,1)\) and (vii) of Proposition \ref{2.2}, we find that
\begin{align}\label{eq0209-5}
t(T\mathbf{1}_E)^*_{\mathcal{H}^\delta_\infty}(t)\ls\int^\infty_0(\mathbf{1}_E)^*_{\mathcal{H}^\delta_\infty}(t)\,dt\sim s, \quad \forall\ t\in(0,\fz).
\end{align}
In addition, since $T$ is of restricted weak type \((\infty,\infty)\) and $\|f\|_{L^\infty(\mathbb{R}^n,\mathcal{H}_{\infty}^{\delta})}=\sup_{t>0}f^{*}_{\mathcal{H}_{\infty}^{\delta}}(t)$,
it follows  that
\begin{align}\label{eq0209-6}
(T\mathbf{1}_E)^{**}_{\mathcal{H}^\delta_\infty}(t)-(T\mathbf{1}_E)^{*}_{\mathcal{H}^\delta_\infty}(t)
\ls \sup_{u\in(0,\infty)}(\mathbf{1}_E)^*_{\mathcal{H}^\delta_\infty}(u)
\sim 1,\quad \forall\ t\in(0,\fz).
\end{align}
From \eqref{eq0209-5} and \eqref{eq0209-6}, we deduce that
\begin{align}\label{eq0209-7}
(T\mathbf{1}_E)^*_{\mathcal{H}^\delta_\infty}(t)\ls
\begin{cases}
1+\log\frac{s}{t}, & t\in(0, s),\\
\ \frac{s}{t}, & t\in\big[s,\infty\big).
\end{cases}
\end{align}
Indeed, \eqref{eq0209-7} holds directly by virtue of \eqref{eq0209-5} when \(t\in[s,\infty)\).
When \(t\in(0,s)\), by \eqref{eq0209-6} and some basic calculations, we find that
\begin{align}\label{eq521-1}
(T\mathbf{1}_E)^{**}_{\mathcal{H}^\delta_\infty}(t)-(T\mathbf{1}_E)^{**}_{\mathcal{H}^\delta_\infty}(s)&=\int_t^s \lf[(T\mathbf{1}_E)^{**}_{\mathcal{H}^\delta_\infty}(u)-(T\mathbf{1}_E)^{*}_{\mathcal{H}^\delta_\infty}(u)\r]\,\frac{du}{u}\\
&\ls\int_{t}^{s}\frac{1}{u}\,du\sim\log\frac{s}{t}\noz.
\end{align}
Setting \(t=s\) in \eqref{eq0209-5} and combining with \eqref{eq0209-6}, we get
\begin{align*}
(T\mathbf{1}_E)^{**}_{\mathcal{H}^\delta_\infty}(s)
\ls (T\mathbf{1}_E)^{*}_{\mathcal{H}^\delta_\infty}(s)+1\ls 1,
\end{align*}
which, together with \eqref{eq521-1} and the fact that
$(T\mathbf{1}_E)^{*}_{\mathcal{H}^\delta_\infty}(t)\leq (T\mathbf{1}_E)^{**}_{\mathcal{H}^\delta_\infty}(t)$,
infers that \eqref{eq0209-7} also holds for \(t\in(0,s)\).
	
By the definition of the operator $S$ and \eqref{eq521-4}, we have
\begin{align*}
		S\lf((\mathbf{1}_E)^*_{\ch^\delta_\infty}\r)(t) =\frac{1}{t}\int_{0}^{t}(\mathbf{1}_E)^*_{\ch^\delta_\infty}(u)\,du+\int_{t}^{\infty}(\mathbf{1}_E)^*_{\ch^\delta_\infty}(u)\,\frac{du}{u}=\begin{cases}
1+\log\frac{s}{t}, & t\in(0, s),\\
\ \frac{s}{t}, & t\in\big[s,\infty\big).
\end{cases}
\end{align*}
Thus, \eqref{eq0209-7} can be written as
\begin{align}\label{eq521-3}
(T\mathbf{1}_E)^{*}_{\ch^\delta_\infty}(t)\ls S\lf((\mathbf{1}_E)^*_{\ch^\delta_\infty}\r)(t),
\end{align}
and
\begin{align}\label{eq521-5}
\frac{1}{t}\int_{0}^{t}S\left((\mathbf{1}_E)^*_{\mathcal{H}^\delta_\infty}\right)(u)\,du=
\begin{cases}
\displaystyle 2+\log\frac{s}{t}, & t\in(0,s),\\[6pt]
\displaystyle \frac{2s}{t}+\frac{s}{t}\log\frac{t}{s}, & t\in[s,\infty).
\end{cases}
\end{align}
Moreover, it follows again from \eqref{eq521-4} that
\begin{align}\label{eq0303-4}
S\lf((\mathbf{1}_E)^{**}_{\ch^\delta_\infty}\r)(t)= \begin{cases}
\displaystyle 2+\log\frac{s}{t}, & t\in(0,s),\\[6pt]
\displaystyle \frac{2s}{t}+\frac{s}{t}\log\frac{t}{s}, & t\in[s,\infty).
\end{cases}
\end{align}
Therefore, combining \eqref{eq521-3}, \eqref{eq521-5} and \eqref{eq0303-4}, we conclude that
\begin{align*}
(T\mathbf{1}_E)^{**}_{\mathcal{H}^\delta_\infty}(t)
\ls\frac{1}{t}\int_{0}^{t}S\left((\mathbf{1}_E)^*_{\mathcal{H}^\delta_\infty}\right)(u)\,du
\sim S\left((\mathbf{1}_E)^{**}_{\mathcal{H}^\delta_\infty}\right)(t),\quad \forall\ t\in(0,\fz),
\end{align*}
which implies that inequality \eqref{eq520-1} holds.

\textbf{Step ii)} We prove that, for any simple function $f\in \cs(\rn,\ch_\fz^\delta)$,
\begin{align}\label{eq521-6}
(Tf)^{**}_{\mathcal{H}^\delta_\infty}(t)\ls S\left(f^{**}_{\mathcal{H}^\delta_\infty}\right)(t)\quad \forall\ t\in(0,\infty).
\end{align}
To this end, it suffices to show that, for any non-negative simple function $g$,
\begin{align}\label{eq521-7}
(Tg)^{**}_{\mathcal{H}^\delta_\infty}(t)\ls S\left(g^{**}_{\mathcal{H}^\delta_\infty}\right)(t)\quad \forall\ t\in(0,\infty).
\end{align}
Indeed, for an arbitrary simple function $f$, we can decompose it as the difference of two non-negative simple functions $f_1$ and $f_2$, i.e., $f=f_1-f_2$. It is obvious that for all $t\in(0,\infty)$,
\[
(f_1)^{**}_{\mathcal{H}^\delta_\infty}(t)\leq f^{**}_{\mathcal{H}^\delta_\infty}(t),\quad (f_2)^{**}_{\mathcal{H}^\delta_\infty}(t)\leq f^{**}_{\mathcal{H}^\delta_\infty}(t).
\]
Combining the above estimates, \eqref{eq521-7}, Corollary \ref{cor510-1} and the sublinearity of operator $T$, we immediately conclude that \eqref{eq521-6} holds.

Next, we prove \eqref{eq521-7}. Let
\[
g:=\sum_{k=1}^{N}a_k\mathbf{1}_{E_k}
\]
be a non-negative finite simple function, where $N\in\mathbb{N}$, $a_1>a_2>\cdots>a_N>0$, and
$\{E_k\}_{k=1}^N$ are pairwise disjoint subsets of $\rn$ with $\mathcal{H}_{\infty}^{\delta}(E_k)<\infty$. For each $k=1,2,\dots,N$, set $F_k:=\bigcup_{j=1}^{k}E_j$ and $F_0:=\emptyset$. Define $b_k:=a_k-a_{k+1}$ for $k=1,2,\dots,N-1$ and $b_N:=a_N$. Then we can rewrite $g$ as
\begin{align}\label{eq521-9}
g=\sum_{k=1}^{N}b_k\mathbf{1}_{F_k}.
\end{align}
Moreover, in view of \eqref{eq521-8}, we have, for any $t\in(0,\infty)$,
\[
g^{*}_{\mathcal{H}_{\infty}^{\delta}}(t)
=\sum_{k=1}^{N}a_k\mathbf{1}_{\left[\mathcal{H}_{\infty}^{\delta}(F_{k-1}),\mathcal{H}_{\infty}^{\delta}(F_k)\right)}(t)
=\sum_{k=1}^{N}b_k\mathbf{1}_{\left[0,\mathcal{H}_{\infty}^{\delta}(F_k)\right)}(t)
=\sum_{k=1}^{N}b_k\left(\mathbf{1}_{F_k}\right)^{*}_{\mathcal{H}_{\infty}^{\delta}}(t),
\]
and then
\[
g^{**}_{\mathcal{H}_{\infty}^{\delta}}(t)=\sum_{k=1}^{N}b_k\left(\mathbf{1}_{F_k}\right)^{**}_{\mathcal{H}_{\infty}^{\delta}}(t),\quad \forall\,t\in(0,\infty).
\]
Thus, by Corollary \ref{cor510-1}, the sublinearity of operator $T$, the linearity of operator $S$, \eqref{eq520-1} and \eqref{eq521-9}, we deduce that, for any $t\in(0,\fz)$,
\begin{align*}
(Tg)^{**}_{\mathcal{H}^\delta_\infty}(t)
&\leq \sum_{k=1}^{N}b_{k}\left(T\mathbf{1}_{F_{k}}\right)^{**}_{\mathcal{H}^\delta_\infty}(t)
\ls\sum_{k=1}^{N}b_{k}S\left(\left(\mathbf{1}_{F_{k}}\right)^{**}_{\mathcal{H}^\delta_\infty}\right)(t)\\
&\sim \left(\sum_{k=1}^{N}b_{k}\left(\mathbf{1}_{F_{k}}\right)^{**}_{\mathcal{H}^\delta_\infty}\right)(t)
\sim S\left(g^{**}_{\mathcal{H}^{\delta}_{\infty}}\right)(t).
\end{align*}
This implies that \eqref{eq521-7} is valid for any non-negative simple function $g$ and
hence \eqref{eq521-6} also holds true.

\textbf{Step iii)} We proceed to prove \eqref{eq0209-4}.
For any $p\in(1,\infty)$, we fix some constant $q\in(p,\infty)$.
Let the operator $\mathcal R$ be as in Lemma \ref{lm520-1} with $(p_1,p_2)=(1,q)$.
Then by Lemma \ref{lm520-1} and Remark \ref{2.2}(ix), we get
\begin{align}\label{eq0521-10} \int^{\infty}_0\lf[\mathcal{R}\left(f^{**}_{\mathcal{H}_{\infty}^{\delta}}\right)(t)\r]^p\,dt\ls\int^{\infty}_0\lf(f_{\mathcal{H}^{\delta}_{\infty}}^{**}(t)\r)^p\,dt\ls\int_\rn|f(x)|^p\, d\ch^\delta_\infty.
\end{align}
Notice that, for any $t\in(0,\infty)$,
\[
S\left(f^{**}_{\mathcal{H}^{\delta}_{\infty}}\right)(t)
\le \frac{q}{t}\int_0^t f^{**}_{\mathcal{H}_{\infty}^{\delta}}(s)\,ds
+\int_t^\infty \frac{1}{s}f^{**}_{\mathcal{H}_{\infty}^{\delta}}(s)\,ds
\le q
\mathcal{R}\left(f^{**}_{\mathcal{H}_{\infty}^{\delta}}\right)(t).\]
Then, by (vii) and (ix) of Proposition \ref{2.2}, \eqref{eq521-6} and \eqref{eq0521-10}, we conclude that
\begin{align*}
\int_\rn |Tf(x)|^p\, d\ch^\delta_\infty&\ls\int_0^{\fz} \lf[(Tf)^{*}_{\ch^\delta_\infty}(t)\r]^p\,dt\ls \int^{\infty}_0\lf[(Tf)^{**}_{\ch^\delta_\infty}(t)\r]^p\,dt\\
&\ls \int^{\infty}_0\lf[S\lf(f_{\mathcal{H}^{\delta}_{\infty}}^{**}\r)(t)\r]^p\,dt\ls \int^{\infty}_0\lf[\mathcal{R}\left(f^{**}_{\mathcal{H}_{\infty}^{\delta}}\right)(t)\r]^p\,dt\ls\int_\rn|f(x)|^p\, d\ch^\delta_\infty.
\end{align*}
This implies that \eqref{eq0209-4} holds for all $\cs(\rn,\ch_\fz^\delta)$.
	
\textbf{Step iv)} We show that if $T$ is linear, then $T$ admits a unique extension to a bounded linear operator on $L^p(\mathbb{R}^n,\mathcal{H}^\delta_\infty)$ for any $p\in(1,\infty)$.
For any $f\in L^p(\mathbb{R}^n,\mathcal{H}^\delta_\infty)$, by Lemma \ref{lm518-1}, there exists a sequence $\{f_k\}_{k}\subset \mathcal{S}(\rn,\mathcal{H}_{\infty}^{\delta})$ such that
\begin{align}\label{eq521-10}
\lim_{k\to\infty}\|f-f_k\|_{L^p(\mathbb{R}^n,\mathcal{H}_{\infty}^{\delta})}=0.
\end{align}
Then it is obvious that
\[
\lim_{\substack{k\to\infty\\m\to\infty}}\|f_k-f_m\|_{L^p(\mathbb{R}^n,\mathcal{H}_{\infty}^{\delta})}=0.
\]
Subsequently, by virtue of the linearity of operator $T$ and \eqref{eq0209-4}, we deduce that
\begin{align*}
\lim_{\substack{k\to\infty\\m\to\infty}}\left\|Tf_{k}-Tf_{m}\right\|_{L^{p}(\mathbb{R}^{n},\mathcal{H}^{\delta}_{\infty})}
=\lim_{\substack{k\to\infty\\m\to\infty}}\left\|T(f_{k}-f_{m})\right\|_{L^{p}(\mathbb{R}^{n},\mathcal{H}^{\delta}_{\infty})}
\ls\lim_{\substack{k\to\infty\\m\to\infty}}\left\|f_{k}-f_{m}\right\|_{L^{p}(\mathbb{R}^{n},\mathcal{H}^{\delta}_{\infty})}=0.
\end{align*}
Thus, applying Lemma \ref{lm518-2}, there exists a function denoted by $Tf$ with $Tf\in L^p(\mathbb{R}^n,\mathcal{H}_{\infty}^{\delta})$ such that
\begin{align}\label{eq521-11}
\lim_{k\to\infty}\|Tf_k-Tf\|_{L^p(\mathbb{R}^n,\mathcal{H}_{\infty}^{\delta})}=0.
\end{align}
It is easy to verify that the operator defined above is well-defined.
Furthermore, we have, for any $k\in\mathbb{N}$,
\begin{align*}
\|Tf\|_{L^p(\mathbb{R}^{n},\mathcal{H}^{\delta}_{\infty})}
&\leq  \|Tf_k-Tf\|_{L^p(\mathbb{R}^{n},\mathcal{H}^{\delta}_{\infty})}+\|Tf_{k}\|_{L^p(\mathbb{R}^{n},\mathcal{H}^{\delta}_{\infty})}\\
&\leq  \|Tf_k-Tf\|_{L^p(\mathbb{R}^{n},\mathcal{H}^{\delta}_{\infty})}+\|f_{k}\|_{L^p(\mathbb{R}^{n},\mathcal{H}^{\delta}_{\infty})}\\
&\lesssim\|Tf_k-Tf\|_{L^p(\mathbb{R}^{n},\mathcal{H}^{\delta}_{\infty})}
+\|f_{k}-f\|_{L^p(\mathbb{R}^{n},\mathcal{H}^{\delta}_{\infty})}+
\|f\|_{L^p(\mathbb{R}^{n},\mathcal{H}^{\delta}_{\infty})}.
\end{align*}
Taking $k\to\infty$ and combining with \eqref{eq521-10} and \eqref{eq521-11}, we immediately conclude that $T$ is a bounded linear operator on $L^p(\mathbb{R}^n,\mathcal{H}_{\infty}^{\delta})$.
In addition, the uniqueness of the extended operator $T$ is obvious. This completes the proof of Theorem \ref{tm0209-1}.
\end{proof}

\section{Proof of Theorem \ref{tm1127-1}\label{Sec4}}

Let $\delta\in(0,n]$ and $Q$ be a given cube of $\rn$. We recall the  \emph{sharp Hausdorff content maximal operator} on $Q$ is defined
by setting, for any $f\in L^1(Q, \ch_{\fz}^{\dz})$ and any $x\in\rn$,
\[{\cm}^{\#, Q}_{\ch_{\fz}^{\dz}}f(x):=\sup_{P\subset Q, P\ni x}\inf_{c\in\mathbb R}\frac{1}{\ch_{\fz}^{\dz}(P)}\int_{P}|f(y)-c|\,d\ch_{\fz}^{\dz}.\]

\begin{rem}\label{rem-1.1}
Let $\dz\in (0, n]$ and $Q$ be a given cube of $\rn$.
Then it is obviously to know that, $f\in {\rm {BMO}}^\delta(Q)$,
\[\|f\|_{{\rm {BMO}}^\delta(Q)}=\sup_{x\in Q}{\cm}^{\#, Q}_{\ch_{\fz}^{\dz}}f(x).\]
\end{rem}

For a given left-open and right-closed cube $Q$ in $\mathbb R^n$,
let $\mathcal D(Q)$ be the set of all dyadic cubes generated by $Q$.
Denote by $\mathcal D$ the collection of all classical dyadic cubes in $\mathbb R^n$.
It is easy to see that if $Q$ is just a dyadic cube in $\mathbb R^n$, then $\mathcal D(Q)=\mathcal D$.
Furthermore, we denote by $\mathcal D_0(Q)$ the family of all dyadic subcubes of $Q$, namely
$$\mathcal D_0(Q):=\{P\in \mathcal D(Q): P\subset Q\}.$$

For a set $E\subset\mathbb R^n$ and $\delta \in (0,n]$, its \emph{inner dyadic Hausdorff content} $\widetilde{\mathcal H}_{\circ}^{\delta}(E)$ (see, for example, \cite{Ada15}) is defined as
\begin{align*}
\widetilde{\mathcal H}_\circ^{\delta}(E)
:=\inf\left\{\sum_i \big[l(Q_i)\big]^{\delta}:\
E\subset \left(\bigcup_i Q_i\right)^\circ,\ Q_i\in \mathcal D,\ \forall\, i\right\},
\end{align*}
where for any set $G\subset \rn$, $G^\circ$ stands for the interior of $G$,
and $l(Q_i)$ represents the side length of cube $Q_i$. Moreover, for any $E\subset\mathbb R^n$,
\begin{equation}\label{eq0723c}
\mathcal H^{\delta}_{\fz}(E)\le \widetilde{\mathcal H}_{\circ}^{\delta}(E)\le K\mathcal H^{\delta}_{\fz}(E).
\end{equation}
Here the constant $K$ is as in \eqref{eq0723a}.

The following lemma is a partly extension of \cite[Lemma 3.2]{BDS-1} or \cite[Lemma 1.1]{BS-1}.
\begin{lem}\label{lm1121-1}
Let $\dz\in(0,n]$ and $Q$ be a given cube of $\rn$. If $O$ is a relatively open subset of $Q$ such that $\mathcal{H}_{\infty}^{\delta}(O)< \frac{1}{2}\mathcal{H}_{\infty}^{\delta}(Q)$. Then there exists a family $\{Q_j\}_{j\in\nn}\subset \mathcal D_0(Q)$ of non-overlapping dyadic subcubes of $Q$ such that
\begin{enumerate}
\item[{\rm(i)}]
$O\subset\bigcup_{j\in\nn} Q_{j}\subset Q;$
\item[{\rm(ii)}]
 For any $j\in\nn$, we have
$$\mathcal{H}^{\delta}_{\infty}(O\cap Q_{j})< \frac{1}{2}\mathcal{H}^{\delta}_{\infty}(Q_{j})\leq 2^{\dz}\mathcal{H}^{\delta}_{\infty}(O\cap Q_{j}).$$
\end{enumerate}
\end{lem}

\begin{proof}
For any fixed $x\in O$, let $Q_x^\ast$ be the smallest dyadic subcube of $Q_x\in \cd_0(Q)$ satisfying
\begin{align}\label{eq12-3a}
x\in Q_x \quad and \quad \mathcal{H}^{\delta}_{\infty}(O\cap Q_x)< \frac{1}{2}\mathcal{H}^{\delta}_{\infty}(Q_x).
\end{align}
We claim that such a $Q_x^\ast$ exists. Indeed, since $O$ is a relatively open subset of $Q$ and $\ch_{\fz}^{\dz}(O)<\frac{1}{2}\mathcal{H}^{\delta}_{\infty}(Q)$, it follows that
the cube $Q$ itself satisfies the requirement \eqref{eq12-3a}.
On the other hand, from the open condition of $O$, we deduce that, when such subcube $Q_x$ small enough,
$Q_x \subset O$ and hence $\mathcal{H}^{\delta}_{\infty}(O\cap Q_x)=\ch_{\fz}^{\dz}(Q_x)> \frac{1}{2}\mathcal{H}^{\delta}_{\infty}(Q_x)$.
Therefore, not all small cubes containing $x$ satisfy the requirement \eqref{eq12-3a}, which implies the existence of $Q_x^\ast$.

Next, we prove that
\begin{align}\label{eq1121-2}
\mathcal{H}^{\delta}_{\infty}(Q_x^\ast)\leq 2^{\delta+1} \mathcal{H}^{\delta}_{\infty}(O\cap Q_x^\ast).
\end{align}
We divide $Q_x^\ast$ into $2^n$ dyadic subcubes with side length $\frac{1}{2} l(Q_x^\ast)$,
denoted by $\{Q_{x,j}^\ast\}_{j=1}^{2^n}$. Without loss of generality, we may assume $x \in Q_{x,1}^\ast$.
By the minimality of $Q_x^\ast$, we know that the subcube $Q_{x,1}^\ast$ does not satisfy the requirement \eqref{eq12-3a}, i.e. $\mathcal{H}^{\delta}_{\infty}(O\cap Q_{x,1}^\ast)\ge \frac{1}{2}\mathcal{H}^{\delta}_{\infty}(Q_{x,1}^\ast)$.
Consequently,
\[\mathcal{H}^{\delta}_{\infty}(Q_{x}^\ast)=2^{\dz}\mathcal{H}^{\delta}_{\infty}(Q_{x,1}^\ast)\leq
2^{\dz+1}\mathcal{H}^{\delta}_{\infty}(O\cap Q_{x,1}^\ast)
\le2^{\dz+1}\mathcal{H}^{\delta}_{\infty}(O\cap Q_{x}^\ast).\]
Let
$$G:=\{Q_{x}^\ast:\ \forall\,x\in O\}.$$
Then $G$ is either a countable or a finite set due to that dyadic cubes are countable under the standard grid.
We let $\{Q_j\}_{j}$ be the maximal dyadic cubes of $G$, i.e., if $R$, $R'\in \{Q_j\}_j$, then $R\cap R'=\emptyset$, and for any $Q_x^\ast \in G$, there exists some $Q_j$ such that $Q_x^\ast \subseteq Q_j$. Therefore,
\[O\subset \bigcup_{x\in O}Q_x^\ast=\bigcup_{R\in G}R=\bigcup_{j\in\nn} Q_j.\]
For any $j\in\nn$, since $Q_j=Q_x^\ast$ for some $x\in O$, it follow that $\mathcal{H}^{\delta}_{\infty}(O\cap Q_{j})< \frac{1}{2}\mathcal{H}^{\delta}_{\infty}(Q_{j})$, and by \eqref{eq1121-2}, $\mathcal{H}^{\delta}_{\infty}(Q_j)\leq2^{\delta+1} \mathcal{H}^{\delta}_{\infty}(O\cap Q_j)$. Thus, condition (ii) is satisfied.
The finishes the proof of Lemma \ref{lm1121-1}.
\end{proof}

The following lemma is a variant of \cite[Lemma 2]{0923-7} (see also \cite{0702-4} for the weighted case).
\begin{lem}\label{lm0903}
Let $\dz\in(0,n]$ and $Q$ be a given cube of $\rn$. Suppose $\{Q_{j}\}_{j\in\nn}\subset \mathcal D_0(Q)$ is a family of non-overlapping dyadic subcubes of $Q$. Then there exists a subfamily $\{Q_{j_v}\}_{v\in \Lambda}$ of $\{Q_j\}_{j\in\nn}\subset \mathcal D_0(Q)$ and a sequence $\{Q^*_k\}_{k\in\digamma}\subset \mathcal D_0(Q)$ of non-overlapping dyadic subcubes of $Q$, such that
\begin{enumerate}
\item[{\rm(i)}] $$\bigcup_{j\in\nn}Q_j\subset \lf(\bigcup_{v\in\Lambda}Q_{j_v}\r)\bigcup\lf(\bigcup_{k\in \digamma}Q^*_k\r);$$
\item[{\rm(ii)}]
for any $P\in\mathcal{D}_0(Q)$,
$$\sum_{Q_{j_v}\subset P} \ch_\fz^\delta(Q_{j_v})\le 2 \ch_\fz^\delta(P);$$
\item[{\rm(iii)}]
for any $k\in \digamma$, there exists some $j\in\nn$ such that $Q_j\subset Q^*_k$, and
$${\ch}_\fz^{\delta} \lf(Q^*_k\r)
\le \sum_{Q_{j_v}\subset Q^*_k} {\ch}_\fz^{\delta}(Q_{j_v}).$$
\end{enumerate}
\end{lem}

The following proposition is derived from variants of \cite[(6) and (7)]{0923-7}.
\begin{prop}\label{prop121-1}
Let $\dz\in(0,n]$ and $Q$ be a given cube of $\rn$. Let $\{Q_{j}\}_{j\in\nn}\subset \mathcal D_0(Q)$ be a family of non-overlapping dyadic subcubes of $Q$ satisfying that, for any $P\in \mathcal{D}_0(Q)$,
\begin{align*}
\sum_{Q_{j}\subset P} \ch_\fz^\delta(Q_{j})\le 2 \ch_\fz^\delta(P).
\end{align*}
Then there exists a positive constant $L=L(n, \dz)$ such that, for any $f\in L^1({\cup_{j\in \nn} Q_{j}},{\ch}_\fz^{\delta})$,
\begin{equation*}
\sum_{j\in \nn}\int_{Q_{j}}|f(x)|\,d{\ch}_\fz^{\delta}\le C\int_{\bigcup_{j\in\nn} Q_{j}}|f(x)|\,d{\ch}_\fz^{\delta}.
\end{equation*}
\end{prop}

Now we are ready to show Theorem \ref{tm1127-1}.
\begin{proof}[Proof of Theorem \ref{tm1127-1}]
Let $Q$ be a given cube of $\rn$ and $f\in {\rm BMO}^\delta(Q)$.
According to Remark \ref{rem-1.1}(iii), we may assume that $f$ is a non-negative function on $Q$.

We claim that there exists a positive constant $C$ such that, for any $t\in (0, \frac{1}{6K^2}\mathcal{H}^{\delta}_{\infty}(Q))$,
\begin{align}\label{eq1123-1}
f^{**, Q}_{\ch_\fz^{\dz}}f(t)-f^{*, Q}_{\ch_\fz^{\dz}}f(t)\leq C({\cm}^{\#, Q}_{\ch_{\fz}^{\dz}}f)^{*, Q}_{\ch_\fz^{\dz}}(t),
\end{align}
where $K$ is as in \eqref{eq0723a}.

Using this claim, we conclude that, when $t\in (0,\frac{1}{6K^2}\mathcal{H}^{\delta}_{\infty}(Q))$,
\begin{equation}\label{eq1127-1}
f^{**, Q}_{\ch_\fz^{\dz}}(t)-f^{*, Q}_{\ch_\fz^{\dz}}(t)\ls \|f\|_{{\rm {BMO}}^\delta(Q)}.
\end{equation}
Indeed, since
$$\ch_{\fz}^{\dz}\lf(\lf\{x\in Q: {\cm}^{\#, Q}_{\ch_{\fz}^{\dz}}f(x) >\|f\|_{{\rm {BMO}}^\delta(Q)}\r\}\r)=0$$
due to (i) of Remark \ref{rem-1.1}, it follows that
\[\|f\|_{{\rm {BMO}}^\delta(Q)}\in\lf\{\lambda>0: \ch_{\fz}^{\dz}\lf(\lf\{x\in Q: {\cm}^{\#, Q}_{\ch_{\fz}^{\dz}}f(x) >\lambda\r\}\r)\leq t\r\},\]
which implies that
\begin{align*}
({\cm}^{\#, Q}_{\ch_{\fz}^{\dz}}f)^{*, Q}_{\ch_\fz^\delta}(t)\leq \|f\|_{{\rm {BMO}}^\delta(Q)}.
\end{align*}
Combining this and the above claim, we find that \eqref{eq1127-1} holds true for all
$t\in (0,\frac{1}{6K^2}\mathcal{H}^{\delta}_{\infty}(Q))$.
On the other hand, if $t\in [\frac{1}{6K^2}\mathcal{H}^{\delta}_{\infty}(Q),\fz)$,
then by (vii) of Proposition \ref{2.2} we have
\begin{align*}
f^{**, Q}_{\ch_\fz^{\dz}}(t)-f^{*, Q}_{\ch_\fz^{\dz}}(t)
&\leq f^{**, Q}_{\ch}\lf(\frac{\mathcal{H}^{\delta}_{\infty}(Q)}{6K^2}\r)
=\frac{6K^2}{\ch_{\fz}^{\dz}(Q)}\int_0^{\frac{\mathcal{H}^{\delta}_{\infty}(Q)}{6K^2}}f^{*, Q}_{\ch_\fz^{\dz}}(t)\,dt\\
&\leq \frac{6K^3}{\ch_{\fz}^{\dz}(Q)}\int_{Q}|f(x)|\,d\ch_{\fz}^{\dz}.\noz
\end{align*}
Therefore, the conclusion of Theorem \ref{tm1127-1} holds true.

It remains to prove the above claim \eqref{eq1123-1}.
For any given $t\in (0, \frac{1}{6K^2}\mathcal{H}^{\delta}_{\infty}(Q))$, let
\begin{align*}
E:=\lf\{x\in Q: f(x)>f^{*, Q}_{\ch_\fz^{\dz}}(t)\r\}\quad {\rm{and}} \quad F:=\lf\{x\in Q: {\cm}^{\#, Q}_{\ch_{\fz}^{\dz}}f(x)>({\cm}^{\#, Q}_{\ch_{\fz}^{\dz}}f)^{*, Q}_{\ch_\fz^{\dz}}(t)\r\}.
\end{align*}
Since, for any $\lambda\ge f^{*, Q}_{\ch_\fz^{\dz}}(t)$,
\begin{align*}
\{x\in Q: |f(x)|>\lambda\}=\{x\in E: |f(x)|>\lambda\}
\end{align*}
and
\[\lf\{s\in [0,t):f^{*, Q}_{\ch_\fz^{\dz}}(s)>\lambda\r\}=\lf\{s\in [0,\fz):f^{*, Q}_{\ch_\fz^{\dz}}(s)>\lambda\r\},\]
it follows from Proposition \ref{2.2} (vi) that
\begin{align}\label{eq1126-10}
t\lf({f^{**, Q}_{\ch_\fz^{\dz}}(t)-f^{*, Q}_{\ch_\fz^{\dz}}(t)}\r)
&=\int_0^t\lf[f^{*, Q}_{\ch_\fz^{\dz}}(s)-f^{*, Q}_{\ch_\fz^{\dz}}(t)\r]\,ds\\
&=\int_0^{\fz}\lf|\lf\{s\in [0,t):f^{*, Q}_{\ch_\fz^{\dz}}(s)-f^{*, Q}_{\ch_\fz^{\dz}}(t)>\lambda\r\}\r|\,d\lambda\noz\\
&=\int_{f^{*, Q}_{\ch_\fz^{\dz}}(t)}^{\fz}\lf|\lf\{s\in [0,\fz):f^{*, Q}_{\ch_\fz^{\dz}}(s)>\lambda\r\}\r|\,d\lambda\noz\\
&\sim K\int_{f^{*, Q}_{\ch_\fz^{\dz}}(t)}^{\fz}\ch_{\fz}^{\dz}(\{x\in Q: f(x)>\lambda\})\,d\lambda\noz\\
&=K\int_{E}\lf[f(x)-f^{*, Q}_{\ch_\fz^{\dz}}(t)\r]\,d\ch_{\fz}^{\dz}=:\rm{I}.\noz
\end{align}
To complete the above claim \eqref{eq1123-1}, we need to estimate $\rm{I}$. According to Proposition \ref{2.2}(i), we have
\begin{align*}
\mathcal{H}^{\delta}_{\infty}(E)\leq Kt\quad {\rm and}\quad
\mathcal{H}^{\delta}_{\infty}(F)\leq Kt,
\end{align*}
and hence by \eqref{eq0723c}
\begin{align}\label{eq1124-1}
\widetilde{\ch}_{\circ}^{\dz}(E\cup F)
\le K\mathcal{H}^{\delta}_{\infty}(E\cup F)
\leq K\lf(\mathcal{H}^{\delta}_{\infty}(E)+\mathcal{H}^{\delta}_{\infty}(F)\r)\leq 2K^2t.
\end{align}
By the definition of the Hausdorff content
$\widetilde{\ch}^{\dz}_{\circ}$, we know that there exists a family of dyadic cubes $\{P_{j}\}_{j\in\nn}$ of $\rn$ such that $E\cup F \subset (\cup_{j\in\nn}P_j)^{\circ}$ and
\begin{align}\label{eq1124-2}
\sum_{j\in\nn}l(P_{j})^{\dz}\leq \widetilde{\ch}_{\circ}^{
\dz}(E\cup F)+K^2t.
\end{align}
Let $O:=(\cup_{j\in\nn}{P}_{j})^{\circ}\cap Q$. Then $O$ is a relatively open subset of $Q$ and $E\cup F \subset O \subset Q$. Moreover, by \eqref{eq1124-2} and \eqref{eq1124-1},
\begin{align*}
\mathcal{H}^{\delta}_{\infty}(O)
\leq \ch_{\fz}^{\dz}\lf(\bigcup_{j\in\nn}P_j\r)
\le \sum_{j\in\nn}l(P_{j})^{\dz}
\leq \widetilde{\ch}_{\circ}^{
\dz}(E\cup F)+K^2t\leq 3K^2t
<\frac{1}{2}\ch_{\fz}^{\dz}(Q).
\end{align*}

Now, by Lemma \ref{lm1121-1}, there exists a family $\{Q_j\}_{j\in\nn}$ of non-overlapping dyadic subcubes of $Q$ such that
\begin{equation}\label{eq413-a}
O\subset\bigcup_{j\in\nn} Q_{j}\subset Q
\end{equation}
and
\begin{equation}\label{eq413-b}
\mathcal{H}^{\delta}_{\infty}(O\cap Q_{j})< \frac{1}{2}\mathcal{H}^{\delta}_{\infty}(Q_{j})\leq 2^{\dz}\mathcal{H}^{\delta}_{\infty}(O\cap Q_{j}),\quad \forall\,j\in\nn.
\end{equation}
For such $\{Q_j\}_{j\in\nn}$, applying Lemma \ref{lm0903}, there exists a subfamily $\{Q_{j_v}\}_{v\in\Lambda}$ of $\{Q_j\}_{j\in\nn}$ and a sequence $\{Q^*_k\}_{k\in \digamma}$ of non-overlapping dyadic subcubes of $Q$, such that
\begin{equation}\label{eq413-c}
\bigcup_{j\in\nn}Q_j\subset \lf(\bigcup_{v\in\Lambda}Q_{j_v}\r)\cup\lf(\bigcup_{k\in\digamma}Q^*_k\r),
\end{equation}
\begin{equation*}
\sum_{Q_{j_v}\subset P} \ch_\fz^\delta(Q_{j_v})\le 2 \ch_\fz^\delta(P),\quad \forall\,P\in\mathcal{D}_0(Q),
\end{equation*}
and, for any $k\in \digamma$, there exists some $j\in\nn$ such that
\begin{equation}\label{eq413-e}
Q_j\subset Q^*_k\quad {\rm and}\quad
{\ch}_\fz^{\delta} \lf(Q^*_k\r)
\le \sum_{Q_{j_v}\subset Q^*_k} {\ch}_\fz^{\delta}(Q_{j_v}).
\end{equation}
For any $j\in\nn$, since $\ch_{\fz}^{\dz}(Q_j)\leq \ch_{\fz}^{\dz}(O\cap Q_j)+\ch_{\fz}^{\dz}(O^{\complement}\cap Q_j)$,
it follows from \eqref{eq413-b} that
\begin{align}\label{eq1126-7}
\ch_{\fz}^{\dz}\lf(O^{\complement}\cap Q_j\r)>\frac{1}{2}\mathcal{H}^{\delta}_{\infty}(Q_{j})>0,
\end{align}
which implies $O^{\complement}\cap Q_j\neq \emptyset$. Here and hereafter, $O^\complement:=\rn\setminus O$
and $F^\complement:=\rn\setminus F$.
Hence, for any $v\in\Lambda$, $F^{\complement}\cap Q_{j_v}\neq \emptyset$ and,
moreover, by \eqref{eq413-e}, for any $k\in \digamma$, $F^{\complement}\cap Q^*_k\neq \emptyset$.

For any $v\in\Lambda$ and any $k\in \digamma$, let $x_v\in F^{\complement}\cap Q_{j_v}$ and $y_k\in F^{\complement}\cap Q_{k}^*$.
For any $\varepsilon>0$, there exist constants $c_v,d_k\in \mathbb{R}$ such that
\begin{align}\label{eq1126-5}
\frac{1}{\ch_{\fz}^{\dz}(Q_{j_v})}\int_{Q_{j_v}}|f(x)-c_v|\,d\ch_{\fz}^{\dz}
&\leq \inf_{c\in\mathbb R}\frac{1}{\ch_{\fz}^{\dz}(Q_{j_v})}\int_{Q_{j_v}}|f(x)-c|\,d\ch_{\fz}^{\dz}+\varepsilon\\
&\leq {\cm}^{\#, Q}_{\ch_{\fz}^{\dz}}f(x_v)+\varepsilon
\leq({\cm}^{\#, Q}_{\ch_{\fz}^{\dz}}f)^{*, Q}_{\ch_\fz^{\dz}}(t)+\varepsilon,
\noz
\end{align}
and
\begin{align}\label{eq1126-6}
\frac{1}{\ch_{\fz}^{\dz}(Q^*_{k})}\int_{Q^*_{k}}|f(x)-d_k|\,d\ch_{\fz}^{\dz}
&\leq \inf_{c\in\mathbb R}\frac{1}{\ch_{\fz}^{\dz}(Q^*_{k})}\int_{Q^*_{k}}|f(x)-c|\,d\ch_{\fz}^{\dz}+\varepsilon\\
&\leq {\cm}^{\#, Q}_{\ch_{\fz}^{\dz}}f(y_k)+\varepsilon
\leq({\cm}^{\#, Q}_{\ch_{\fz}^{\dz}}f)^{*, Q}_{\ch_\fz^{\dz}}(t)+\varepsilon.
\noz
\end{align}
Although such $c_v,d_k$ may be not unique, without loss of generality, we select one of them and fix it.

Let
$$A:=\lf\{v\in \Lambda:\ c_v\leq f^{*, Q}_{\ch_\fz^{\dz}}(t)\r\},\ B:=\lf\{v\in \Lambda:\ c_v> f^{*, Q}_{\ch_\fz^{\dz}}(t)\r\},$$
and
$$C:=\lf\{k\in \digamma:\ d_k\leq f^{*, Q}_{\ch_\fz^{\dz}}(t)\r\},\ D:=\lf\{k\in \digamma:\ d_k> f^{*, Q}_{\ch_\fz^{\dz}}(t)\r\}.$$
Then, by \eqref{eq413-a}, \eqref{eq413-c}, \eqref{eq0904} and (ii) of Remark \ref{12-4}, we have
\begin{align}\label{eq1126-9}
{\rm{I}}&\leq \sum_{v\in \Lambda}\int_{E\cap Q_{j_v}}\lf[f(x)-f^{*, Q}_{\ch_\fz^{\dz}}(t)\r]\,d\ch_{\fz}^{\dz}+\sum_{k\in \digamma}\int_{E\cap Q^*_{k}}\lf[f(x)-f^{*, Q}_{\ch_\fz^{\dz}}(t)\r]\,d\ch_{\fz}^{\dz}\\
&\leq \sum_{v\in A}\int_{E\cap Q_{j_v}}\lf[f(x)-c_v\r]\,d\ch_{\fz}^{\dz}+\sum_{v\in B}\int_{E\cap Q_{j_v}}\lf[|f(x)-c_v|+c_v-f^{*, Q}_{\ch_\fz^{\dz}}(t)\r]\,d\ch_{\fz}^{\dz}\noz\\
&\quad +\sum_{k\in C}\int_{E\cap Q^*_k}\lf[f(x)-d_k\r]\,d\ch_{\fz}^{\dz}+\sum_{k\in D}\int_{E\cap Q^*_k}\lf[|f(x)-d_k|+d_k-f^{*, Q}_{\ch_\fz^{\dz}}(t)\r]\,d\ch_{\fz}^{\dz}\noz\\
&\leq \sum_{v\in \Lambda}\int_{Q_{j_v}}|f(x)-c_v|\,d\ch_{\fz}^{\dz}+\sum_{v\in B}\lf[c_v-f^{*, Q}_{\ch_\fz^{\dz}}(t)\r]\ch_{\fz}^{\dz}(E\cap Q_{j_v})\noz\\
&\quad+\sum_{k\in \digamma}\int_{Q^*_k}|f(x)-d_k|\,d\ch_{\fz}^{\dz}+\sum_{k\in D}\lf[d_k-f^{*, Q}_{\ch_\fz^{\dz}}(t)\r]\ch_{\fz}^{\dz}(E\cap Q^*_k)\noz\\
&=:\rm{I_1+I_2+I_3+I_4}.\noz
\end{align}
For ${\rm{I_2}}$, since $E\subset O$, it follows from \eqref{eq413-b} and \eqref{eq1126-7} that
\begin{align}\label{eq1126-8}
{\rm I}_2&\leq\sum_{v\in B}\lf[c_v-f^{*, Q}_{\ch_\fz^{\dz}}(t)\r]\ch_{\fz}^{\dz}(O\cap Q_{j_v})\leq \sum_{v\in B}\int_{O^{\complement}\cap Q_{j_v}}\lf[c_v-f^{*, Q}_{\ch_\fz^{\dz}}(t)\r]\,d\ch_{\fz}^{\dz}\\
&\leq \sum_{v\in B}\int_{O^{\complement}\cap Q_{j_v}}\lf[c_v-f(x)\r]\,d\ch_{\fz}^{\dz}\leq {\rm{I_1}}.\noz
\end{align}
For ${\rm{I_4}}$, applying Proposition \ref{prop121-1} to cubes $\{Q_{j_v}\}_{Q_{j_v}\subset Q_k^*}$ with any $k\in D$,
 by \eqref{eq413-b} and \eqref{eq1126-7}, we find that
\begin{align*}
{\rm I}_4&\leq\sum_{k\in D}\lf[d_k-f^{*, Q}_{\ch_\fz^{\dz}}(t)\r]\ch_{\fz}^{\dz}(Q^*_k)\leq\sum_{k\in D}\lf[d_k-f^{*, Q}_{\ch_\fz^{\dz}}(t)\r]\sum_{Q_{j_v}\subset Q^*_k} {\ch}_\fz^{\delta}(Q_{j_v})\\
&\leq2^{\dz+1}\sum_{k\in D}\lf[d_k-f^{*, Q}_{\ch_\fz^{\dz}}(t)\r]\sum_{Q_{j_v}\subset Q^*_k} {\ch}_\fz^{\delta}(O\cap Q_{j_v})\\
&\leq2^{\dz+1}\sum_{k\in D}\sum_{Q_{j_v}\subset Q^*_k}\int_{Q_{j_v}}\lf[d_k-f^{*, Q}_{\ch_\fz^{\dz}}(t)\r]\mathbf{1}_{O^{\complement}\cap Q^*_k}(x)\,d{\ch}_\fz^{\delta}\\
&\ls2^{\dz+1}\sum_{k\in D}\int_{O^{\complement}\cap Q^*_{k}}\lf[d_k-f^{*, Q}_{\ch_\fz^{\dz}}(t)\r]\,d{\ch}_\fz^{\delta}\\
&\ls2^{\dz+1}\sum_{k\in D}\int_{O^{\complement}\cap Q^*_{k}}\lf[d_k-f(x)\r]\,d{\ch}_\fz^{\delta}
\ls 2^{\dz+1}{\rm I}_3.
\end{align*}
From this, combining \eqref{eq1126-8}, \eqref{eq1126-9}, \eqref{eq1126-5}, \eqref{eq1126-6}, \eqref{eq413-b}
and Proposition \ref{prop121-1} again, we deduce that
\begin{align*}
{\rm I}&\ls 2{\rm I}_1+(2^{\dz+1}+1){\rm I}_3\\
&\ls \lf(({\cm}^{\#, Q}_{\ch_{\fz}^{\dz}}f)^{*, Q}_{\ch_\fz^{\dz}}(t)+\varepsilon\r)\lf\{\sum_{v\in \Lambda}\ch_{\fz}^{\dz}(Q_{j_v})+\sum_{k\in \digamma}\ch_{\fz}^{\dz}(Q^*_{k})\r\}\\
&\ls \lf(({\cm}^{\#, Q}_{\ch_{\fz}^{\dz}}f)^{*, Q}_{\ch_\fz^{\dz}}(t)+\varepsilon\r)\lf\{\sum_{v\in \Lambda}\ch_{\fz}^{\dz}(Q_{j_v})+\sum_{k\in \digamma}\sum_{Q_{j_v}\subset Q^*_k} {\ch}_\fz^{\delta}(Q_{j_v})\r\}\\
&\ls\sum_{v\in \Lambda}\ch_{\fz}^{\dz}(Q_{j_v})\lf(({\cm}^{\#, Q}_{\ch_{\fz}^{\dz}}f)^{*, Q}_{\ch_\fz^{\dz}}(t)+\varepsilon\r)\ls\sum_{v\in \Lambda}\int_{Q_{j_v}}\mathbf{1}_{O}(x)\,d\ch_{\fz}^{\dz}\lf(({\cm}^{\#, Q}_{\ch_{\fz}^{\dz}}f)^{*, Q}_{\ch_\fz^{\dz}}(t)+\varepsilon\r)\\
&\ls\int_{\cup_{v\in \Lambda}Q_{j_v}}\mathbf{1}_{O}(x)\,d\ch_{\fz}^{\dz}\lf(({\cm}^{\#, Q}_{\ch_{\fz}^{\dz}}f)^{*, Q}_{\ch_\fz^{\dz}}(t)+\varepsilon\r)\ls\ch_{\fz}^{\dz}(O)\lf(({\cm}^{\#, Q}_{\ch_{\fz}^{\dz}}f)^{*, Q}_{\ch_\fz^{\dz}}(t)+\varepsilon\r)\\
& \ls t\lf(({\cm}^{\#, Q}_{\ch_{\fz}^{\dz}}f)^{*, Q}_{\ch_\fz^{\dz}}(t)+\varepsilon\r).
\end{align*}
By this and the arbitrariness of \(\varepsilon>0\), together with \eqref{eq1126-10}, we immediately obtain the desired claim \eqref{eq1123-1}. This finishes the proof of Theorem \ref{tm1127-1}.
\end{proof}

\section{Maximal Operator on $wL^\fz(Q,\ch_\fz^\delta)$\label{Sec5}}

Let $\dz\in (0, n]$ and $Q$ be a given cube of $\rn$. If $f\in L^1(Q, \ch_{\fz}^{\dz})$,
then its \emph{local Hausdorff content Hardy-Littlewood maximal function} is defined by
\begin{align*}
\cm^{Q}_{\widetilde\ch_{\fz}^{\dz}}f(x):=\sup_{P\ni x\ and\ P\subset Q}\frac{1}{\widetilde\ch_{\fz}^{\dz}(P)}\int_{P}|f(y)|\,d\widetilde\ch_{\fz}^{\dz},\quad \forall x\in Q,
\end{align*}
where the supremum is taken over all cube $P\subset Q$ containing $x$.
Particularly, when $Q=\rn$, it is the Hausdorff content maximal operator, denoted simply by $\cm_{\widetilde\ch_{\fz}^{\dz}}$.

\begin{rem}\label{rem122-2} Let $\dz\in (0, n]$ and $Q$ be a given cube of $\rn$.
\begin{enumerate}
\item[{\rm(i)}]
If $f$ is a function defined on $Q$, then by \eqref{eq0723a}, it is clear to know that
\[\cm^{Q}_{\widetilde\ch_{\fz}^{\dz}}f(x)\leq \cm_{\widetilde\ch_{\fz}^{\dz}}f(x)\leq K\cm_{\ch_{\fz}^{\dz}}f(x),\quad \forall x\in Q, \]
where we regard $f$ as zero on $\rn\setminus Q$. Furthermore, according to \cite[Theorem 1.2]{0923-3}, we know that there exists a positive constant $C$ such that, for any $g\in L^1(\rn, \ch_{\fz}^{\dz})$,
\[\ch_{\fz}^{\dz}\lf(\lf\{x\in \rn: \cm_{\ch_{\fz}^{\dz}}g(x)>t\r\}\r)\leq \frac{C}{t}\|g\|_{L^1({\rn, \ch_{\fz}^{\dz}})},\quad \forall t\in (0,\fz). \]
Therefore,
\begin{align}\label{eq122-5}
\ch_{\fz}^{\dz}\lf(\lf\{x\in Q: \cm^{Q}_{\widetilde\ch_{\fz}^{\dz}}f(x)>t\r\}\r)\leq \frac{KC}{t}\|f\|_{L^1({Q, \ch_{\fz}^{\dz}})},\quad \forall t\in (0,\fz),
\end{align}
where $K$ is as in \eqref{eq0723a}.

\item[{\rm(ii)}]
It is not difficult to find that, for any $f\in L^{\fz}(Q, \ch_{\fz}^{\dz})$,
\begin{align}\label{eq122-6}
\|\cm^{Q}_{\widetilde\ch_{\fz}^{\dz}}f\|_{L^{\fz}(Q, \ch_{\fz}^{\dz})}\leq \|f\|_{L^{\fz}(Q, \ch_{\fz}^{\dz})}.
\end{align}
\item[{\rm(iii)}]
Following the argument of \cite[Lemma 4.3]{0702-4} and optimizing the constant estimates in derivation, we obtain an improved inequality:
For any $f\in L^1(Q,\ch_{\fz}^{\dz})$, the inequality
\[\cm^{Q}_{\widetilde\ch_{\fz}^{\dz}}f(x)\ge |f(x)|\]
holds for $\widetilde\ch_{\fz}^{\dz}$-almost every $x\in Q$.
\end{enumerate}
\end{rem}

\begin{thm}\label{tm1127-2}
Let $\dz\in(0, n]$ and $Q$ be a given dyadic cube of $\rn$.
Then there exists a positive constant $C$ such that, for any $f\in wL^\fz(Q, \ch_{\fz}^{\dz})$,
\begin{align}\label{eq1128-6}
\|\cm^{Q}_{\widetilde\ch_{\fz}^{\dz}}f\|_{wL^\fz(Q, \ch_{\fz}^{\dz})}\leq C\|f\|_{wL^\fz(Q, \ch_{\fz}^{\dz})}.
\end{align}
Moreover, the result also holds true with  $Q$ replaced by $\mathbb{R}^{n}$.
\end{thm}

\begin{proof}
For similarity, we only prove \eqref{eq1128-6}. Since for any $t\in (0,\fz)$, $f^{*, Q}_{\ch_\fz^{\dz}}(t)=|f|^{*, Q}_{\ch_\fz^{\dz}}(t)$, we may assume that $f$ is a non-negative function on $Q$. For any given $t\in (0,\fz)$, we define
\[g(x):=\max\lf(f(x)-f^{*, Q}_{\ch_\fz^{\dz}}(t), 0\r),\quad {\rm{and}} \quad h(x):=\min\lf(f(x), f^{*, Q}_{\ch_\fz^{\dz}}(t)\r),\ \forall\ x\in Q.\]
Then $f=g+h$ and, by \eqref{eq0904}, \[\cm^{Q}_{\widetilde\ch_{\fz}^{\dz}}f(x)\leq \cm^{Q}_{\widetilde\ch_{\fz}^{\dz}}g(x)+\cm^{Q}_{\widetilde\ch_{\fz}^{\dz}}h(x),\ \forall\ x\in Q,\]
which, together with (iv) and (viii) of Proposition \ref{2.2}, implies that
\begin{align}\label{eq1127-4}
(\cm^{Q}_{\widetilde\ch_{\fz}^{\dz}}f)^{*, Q}_{\ch_\fz^{\dz}}(t)\leq (\cm^{Q}_{\widetilde\ch_{\fz}^{\dz}}g)^{*, Q}_{\ch_\fz^{\dz}}(t)+(\cm^{Q}_{\widetilde\ch_{\fz}^{\dz}}h)^{*, Q}_{\ch_\fz^{\dz}}(0).
\end{align}
By \eqref{eq122-5}, we have,
\[\ch_{\fz}^{\dz}\lf(\lf\{x\in Q: \cm^{Q}_{\widetilde\ch_{\fz}^{\dz}}g(x)>\frac{CK}{t}\|g\|_{L^{1}(Q,\mathcal{H}^{\delta}_{\infty})}\r\}\r)\leq t,\]
where the constant $C$ comes from \eqref{eq122-5} and $K$ is as in \eqref{eq0723a}. Thus,
\begin{align}\label{eq1128-3}
(\cm^{Q}_{\widetilde\ch_{\fz}^{\dz}}g)^{*, Q}_{\ch_\fz^{\dz}}(t)=\inf\lf\{\lambda>0: \ch_{\fz}^{\dz}\lf(\lf\{x\in Q: \cm^{Q}_{\widetilde\ch_{\fz}^{\dz}}g(x)>\lambda\r\}\r)\leq t\r\}\leq \frac{CK}{t}\|g\|_{L^{1}(Q,\mathcal{H}^{\delta}_{\infty})}.
\end{align}
Moreover, by \eqref{eq122-6}, we get
\begin{align}\label{eq1128-4}
(\cm^{Q}_{\widetilde\ch_{\fz}^{\dz}}h)^{*, Q}_{\ch_\fz^{\dz}}(0)=\inf\lf\{\lambda>0: \ch_{\fz}^{\dz}\lf(\lf\{x\in Q: \cm^{Q}_{\widetilde\ch_{\fz}^{\dz}}h(x)>\lambda\r\}\r)\leq 0\r\}\leq \|h\|_{L^{\fz}(Q,\mathcal{H}^{\delta}_{\infty})}.
\end{align}
Let $$E:=\lf\{x\in Q: f(x)>f^{*, Q}_{\ch_\fz^{\dz}}(t)\r\}.$$
Then, by \eqref{eq1126-10}, we have
\[\|g\|_{L^{1}(Q,\mathcal{H}^{\delta}_{\infty})}=\int_{E}\lf[f(x)-f^{*, Q}_{\ch_\fz^{\dz}}(t)\r]\,d\ch_{\fz}^{\dz}\leq Kt\lf({f^{**, Q}_{\ch_\fz^{\dz}}(t)-f^{*, Q}_{\ch_\fz^{\dz}}(t)}\r).\]
From this, \eqref{eq1128-3}, \eqref{eq1128-4} and \eqref{eq1127-4}, we deduce that
\[(\cm^{Q}_{\widetilde\ch_{\fz}^{\dz}}f)^{*, Q}_{\ch_\fz^{\dz}}(t)\leq CK^2\lf({f^{**, Q}_{\ch_\fz^{\dz}}(t)-f^{*, Q}_{\ch_\fz^{\dz}}(t)}\r)+\|h\|_{L^{\fz}(Q,\mathcal{H}^{\delta}_{\infty})}\leq CK^2\lf({f^{**, Q}_{\ch_\fz^{\dz}}(t)-f^{*, Q}_{\ch_\fz^{\dz}}(t)}\r)+f^{*, Q}_{\ch_\fz^{\dz}}(t). \]
Hence, $(\cm^{Q}_{\widetilde\ch_{\fz}^{\dz}}f)^{*, Q}_{\ch_\fz^{\dz}}(t)$ is finite and
\begin{align}\label{eq1128-5}
(\cm^{Q}_{\widetilde\ch_{\fz}^{\dz}}f)^{*, Q}_{\ch_\fz^{\dz}}(t)-f^{*, Q}_{\ch_\fz^{\dz}}(t)\leq CK^2\lf({f^{**, Q}_{\ch_\fz^{\dz}}(t)-f^{*, Q}_{\ch_\fz^{\dz}}(t)}\r).
\end{align}
Additionally, based on (iv) of Proposition \ref{2.2} and (iii) of Remark \ref{rem122-2}, we know that
\begin{align*}
(\cm^{Q}_{\widetilde\ch_{\fz}^{\dz}}f)^{*, Q}_{\ch_\fz^{\dz}}(t)-f^{*, Q}_{\ch_\fz^{\dz}}(t)\ge 0.
\end{align*}
Therefore, by \eqref{eq1128-5}, we conclude that
\begin{align*}
(\cm^{Q}_{\widetilde\ch_{\fz}^{\dz}}f)^{**, Q}_{\ch_\fz^{\dz}}(t)-(\cm^{Q}_{\widetilde\ch_{\fz}^{\dz}}f)^{*, Q}_{\ch_\fz^{\dz}}(t)&\leq \lf[(\cm^{Q}_{\widetilde\ch_{\fz}^{\dz}}f)^{**, Q}_{\ch_\fz^{\dz}}(t)-f^{**, Q}_{\ch_\fz^{\dz}}(t)\r]+\lf[f^{**, Q}_{\ch_\fz^{\dz}}(t)-f^{*, Q}_{\ch_\fz^{\dz}}(t)\r]\\
&=\frac{1}{t}\int_{0}^{t}\lf[(\cm^{Q}_{\widetilde\ch_{\fz}^{\dz}}f)^{*, Q}_{\ch_\fz^{\dz}}(s)-f^{*, Q}_{\ch_\fz^{\dz}}(s)\r]\,ds+\lf[f^{**, Q}_{\ch_\fz^{\dz}}(t)-f^{*, Q}_{\ch_\fz^{\dz}}(t)\r]\\
&\leq (CK^2+1)\sup_{0<s\leq t}\lf\{f^{**, Q}_{\ch_\fz^{\dz}}(t)-f^{*, Q}_{\ch_\fz^{\dz}}(t)\r\}\leq \|f\|_{wL^\fz(Q, \ch_{\fz}^{\dz})},
\end{align*}
which implies that \eqref{eq1128-6} holds. We complete the proof of Theorem \ref{tm1127-2}.
\end{proof}

\section*{Data availability}

No data was used for the research described in the article.

\section*{Declaration}

There is no financial or non-financial interests that are directly or indirectly related to
the work submitted for publication.

\section*{Acknowledgements}

This project is supported by the National Natural
Science Foundations of China (12201139), and the Natural Science Foundation of Hunan province (2024JJ3023).

\medskip
\medskip
\noindent Hao Guo, Yangzhi Zhang (Corresponding author) and Ciqiang Zhuo
	
\smallskip
	
\noindent MOE-LCSM, School of Mathematics and Statistics,
Hunan Normal University,
Changsha, Hunan 410081, China
	
\smallskip
\noindent{\it E-mails:}
\texttt{haoguo@hunnu.edu.cn (H. Guo)}

\noindent\phantom {\it E-mail}
\texttt{yzzhang@hunnu.edu.cn (Y. Zhang)}

\noindent\phantom {\it E-mail}
\texttt{cqzhuo87@hunnu.edu.cn (C. Zhuo)}

\medskip
\medskip

\noindent Long Huang

\medskip

\noindent School of Mathematics and Information Science,
Guangzhou University, Guangzhou, 510006, China

\smallskip

\noindent {\it E-mail}: \texttt{longhuang@gzhu.edu.cn}

%
%
%
%

\end{document}